\documentclass{article}

\usepackage{amssymb}
\usepackage{amsmath}
\usepackage{cite}
\usepackage{fancyhdr}
\usepackage{graphicx}
\usepackage{subfigure}
\usepackage{float}
\usepackage{tabularx}
\usepackage{multicol}
\usepackage{multirow}
\usepackage[final]{pdfpages}
\usepackage[left=1in,right=1in,top=1in,bottom=1in]{geometry}
\usepackage{amsthm}
\usepackage{mathrsfs}
\usepackage{pifont}
\usepackage{geometry}
\usepackage{indentfirst}
\newtheorem{theorem}{Theorem}
\renewcommand{\Re}{\operatorname{Re}}
\renewcommand{\Im}{\operatorname{Im}}
\newtheorem{lemma}{Lemma}
\newtheorem{rem}{Remark}

\newtheorem{proposition}{Proposition}

\theoremstyle{noparens}

\title{\bf{Dynamic properties of a class of van der Pol-Duffing oscillators}\author{Yelei Kuang, Xuemei Li
\noindent\footnote{~Corresponding author}
\setcounter{footnote}{-1}
\noindent\footnote{~E-mail: yelei.k@hunnu.edu.cn(Y.Kuang), lixuemei\_1@sina.com(X.Li).}\\\\
\small Key Laboratory of High Performance Computing and Stochastic Information Processing,\\\\
\small Department of Mathematics, Hunan Normal University, Changsha, Hunan 410081, P.R. China.}}
\date{}
\begin{document}
\maketitle

\textbf {Abstract.} In this paper, we study the existence of bifurcation of a van der Pol-Duffing oscillator with quintic terms and its quasi-periodic solutions by means of qualitative and bifurcation theories. Firstly, we analyze the autonomous system and find that it has two kinds of local bifurcations and a global bifurcation: pitchfork bifurcation, Hopf bifurcation, homoclinic bifurcation. It is worth noting that the disappearance of the homoclinic orbit is synchronized with the emergence of a large limit cycle. Then, by discussing the stability of equilibria at infinity and the orientation of the trajectory, the existence and stability of limit circles of the autonomous system are analyzed by combining the Poincar\'{e}-Bendixson theorem and the index theory.  The global phase portrait and the numerical simulation of the autonomous system in different parameter values are given. Finally, the existence of periodic and quasi-periodic solutions to periodic forced system is proved by a KAM theorem.\\


\section{Introduction and main result}
\label{chapter1}
\def\theequation{1.\arabic{equation}}
\setcounter{equation}{0}
\par
In 1927, the physicist Balthasar van der Pol\cite{L1} proposed a oscillator model
\begin{align}\label{1.1.1}
  \ddot{x}-k(1-x^2)\dot{x}+x=0,
\end{align}
 named after van der Pol oscillator, which is used to simulate the limit circle oscillation phenomenon of vacuum tube amplifier, or with a periodic forcing term
\begin{align}\label{1.1.2}
  \ddot{x}-k(1-x^2)\dot{x}+x=b\lambda \sin(\lambda t),
\end{align}
where $k$ is the damping coefficient. When $k=0$, the equation becomes an ordinary simple harmonic vibration equation. When $k<0$, it is damped vibration, and the vibration will gradually decay to 0; when $k>0$, self-excited motion occurs. van der Pol's work has become the basis for many modern theories of nonlinear oscillations. van der Pol oscillator has no analytical solution, but its approximate solutions are obtained by many methods, such as perturbation method\cite{L4}, multi-scale method\cite{L6}, KBM method\cite{L5} and so on.

The Duffing oscillator\cite{L2} is a vibration oscillator describing forced vibration, which is expressed by nonlinear differential equation in the following
\begin{align}\label{1.1.3}
  \ddot{x}+2r\dot{x}+\alpha x+\beta x^3=\delta\cos(\omega t),
\end{align}
where $\gamma$ is the damping coefficient, $\alpha$ is the toughness, $\beta$ is the nonlinearity of the controlling force, $\delta$ is the amplitude of the driving force, and $\omega$ is the angular frequency of the driving force.

When studying some practical problems,  van der Pol equation \eqref{1.1.1} and Duffing equation\eqref{1.1.3} are usually combined together,  called van der Pol-duffing equation
\begin{align}\label{1.1.4}
  \ddot{x}+k(1-x^2)\dot{x}+\alpha x+\beta x^3=\delta\cos(\omega t).
\end{align}
van der Pol-Duffing oscillator\eqref{1.1.4} is one of the most classical mathematical physics models. Because of its rich dynamical properties, it is widely used in physics, biology, engineering and even economics.
There are many generalized models of \eqref{1.1.4} also called van der Pol-Duffing oscillators, which have been extensively studied. Yu, Murza et al.\cite{L24} investigated the dynamics of autonomous ODE systems with D4-symmetry, and applied the  obtained results to two coupled van der Pol-Duffing oscillators with D4-symmetry. Xu and Luo\cite{L25} studied the periodic driven van der Pol-Duffing oscillator and showed the independent bifurcation tree from 2-periodic motion to the coexistence of chaotic and 1-periodic motion by semi-analytical approach. Monsivais, Velazquez et al. \cite{L26} investigated the dynamics of hierarchical weighted networks of van der Pol oscillators. In addition, many scholars have theoretically analyzed the global dynamical properties \cite{L15,L27,L16,L13,L17,L31,L14}, and found attractors and chaotic phenomena \cite{L18,L19,L20} in van der Pol-Duffing oscillator.

Based on the cell mapping method Han and Xu et al.\cite{L8} investigated the dynamical behaviors of a model of van der Pol-Duffing oscillator with a quintic term
\begin{align}\label{1.2.1}
    \left\{
        \begin{array}{ll}
        \dot{x}=y,\\
        \dot{y}=(\mu+x^{2}-x^{4})y+\omega_{0}^{2}x-\lambda x^{3}-\alpha x\cos(\omega t).
        \end{array}
    \right.
\end{align}
where $x, y$ are the state variable, $\mu, \omega_ {0}, \lambda$ and $\alpha$ are real parameters. It is obvious that the system \eqref{1.2.1} is symmetric because the two tracks generated by $(x,y)$ and $(-x,-y)$ are center symmetric about the origin at all times. For the fixed parameters  $\mu=-1, \omega_ {0}=1, \lambda=3, \omega=1$, though numerical simulations the authors showed the change of dynamical behaviors of \eqref{1.2.1} in the region $D=\{(x,y): ~|x|\leq 1.5, |y|\leq 2\}$ as varying the amplitude parameter $\alpha$.

The purpose of this paper is to theoretically analyze the global dynamic properties of \eqref{1.2.1}. For the convenience of writing, $\omega_{0}^2$ and $\lambda$ are marked as $\beta$ and $\epsilon$, respectively. The system \eqref{1.2.1} is written as
\begin{align}\label{1.2.2}
    \left\{
        \begin{array}{ll}
        \dot{x}=y,\\
        \dot{y}=(\mu+x^{2}-x^{4})y+\beta x-\epsilon x^{3}-\alpha x\cos(\omega t),
        \end{array}
    \right.
\end{align}
where $x, y$ are the state variable, $\mu,  \beta,  \epsilon, \alpha $ are real parameter and $\beta \geq0 $.  We will  first analyze global dynamical behaviors of the autonomous system without the external force $(\alpha=0)$, i.e
\begin{align}\label{1.2.3}
    \left\{
        \begin{array}{ll}
        \dot{x}=y,\\
        \dot{y}=(\mu+x^{2}-x^{4})y+\beta x-\epsilon x^{3}.
        \end{array}
    \right.
\end{align}
Obviously,  orbits of the system \eqref{1.2.3} are center symmetric about the origin. By the transformation
\begin{align*}
  x\rightarrow x,~~ y\rightarrow y-F(x),
\end{align*}
where $F(x)=-\mu x-\frac{1}{3}x^3+\frac{1}{5}x^5$. This equation \eqref{1.2.3} can be changed into the $\mathbb{Z}_{2}-$equivariant Li\'{e}nard equation
\begin{align}\label{1.2.5}
  \left\{
        \begin{array}{ll}
        \dot{x}=y-(-\mu x-\frac{1}{3}x^3+\frac{1}{5}x^5)\triangleq y-F(x),\\
        \dot{y}=-(-\beta x+\epsilon x^3)\triangleq -g(x).
        \end{array}
    \right.
\end{align}
We will discuss the existence of limit cycles of \eqref{1.2.3} by the Li\'{e}nard equation \eqref{1.2.5}.

In Section 2, we analyze the stability of equilibrium points $O(0,0)$, $E_{1}(x_{1},0)$, $E_{2}(x_{2},0)$, where $x_{1}=-\sqrt{\frac{\beta}{\epsilon}}$, $x_{2}=\sqrt{\frac{\beta}{\epsilon}}$, of the system \eqref{1.2.3} depending on different parameter values, and find that the pitchfork bifurcation and Hopf bifurcation occur at the equilibrium point $O$ and the equilibrium point $E_1$ (or $E_2$), respectively.  Then the system \eqref{1.2.3} is transformed into a near Hamnilton system, and the branch curve of the homoclinic orbit is obtained by the Melnikov function method.  In Section 3,
 we analyze the dynamical behavior of the system \eqref{1.2.3} at infinity by projecting the system \eqref{1.2.3} onto the ~Poincar\'{e} disk.  Then the existence and non-existence of limit cycles of the system \eqref{1.2.3} in different parameter ranges are proven based on the direction of orbits at infinity, and the global phase diagram is also given. In Section 4, it is first proved that the system \eqref{1.2.2} has a periodic solution with period $\frac{2\pi}{\omega}$ for some parameter regions. Then by using a KAM theorem we obtain that  the system \eqref{1.2.2} possesses a  quasi-periodic solution with two fundamental frequencies generated by the Hopf bifurcation of the system \eqref{1.2.3} with periodic perturbations (i.e., $|\alpha|$ is sufficiently small) near the points $E_{1,2}$. Numerical simulations are also carried out in Section 5.

Our main results are as follows.

\begin{theorem}\label{T2.2.0}
When $\epsilon>0$, System \eqref{1.2.3} goes through the following types of bifurcations in the parameter plane $(\beta,\mu)$ with $\beta\geq 0$

(a) pitchfork bifurcation

supercritical pitchfork bifurcation curve: ~~~$\mu_{-}=\{(\beta,\mu):~\beta=0,~ \mu<0\}$,

subcritical pitchfork bifurcation curve: ~~~$\mu_{+}=\{(\beta,\mu):~\beta=0,~\mu>0\}$;

(b) Hopf bifurcation

bifurcation curve:~~~  $H=\{(\beta,\mu):\mu=\frac{\beta^2-\epsilon\beta}{\epsilon^2},~\beta>0\}$;

(c) Homoclinic orbit bifurcation

bifurcation curve: ~~~$ P=\{(\beta,\mu):\mu=\frac{32\beta^2}{35\epsilon^2}-\frac{4\beta}{5\epsilon}\triangleq\mu_{3},~ \beta>0\}$.\\
The bifurcation diagram is presented in Figure 1.
\end{theorem}

See Propositions 1-2 and Lemmas 2-5 for more details.

\begin{figure}[htbp]
  \centering
  \includegraphics[width=11.7cm]{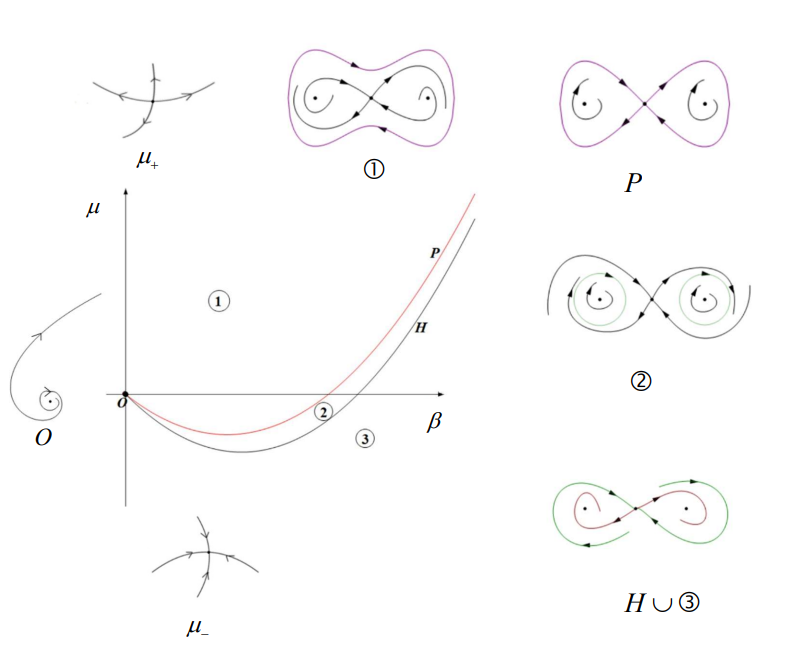}\\
  \caption{when $\epsilon>0$, the bifurcation diagram of System \eqref{1.2.3} \label{F2}}
\end{figure}

\begin{theorem}\label{D2}
When $\epsilon>0$ and $|\alpha|$ is sufficiently small, for most parameters in the sufficiently small right neighborhood of the Hopf bifurcation point $\mu=\mu_{c}(\triangleq\frac{\beta^2-\epsilon\beta}{\epsilon^2})$, the system \eqref{1.2.2} has a quasi-periodic solution with two fundamental frequencies near the point $E_{1}$ (or $E_{2}$).
\end{theorem}

See Theorem 3 for more details.

\section{Local dynamics}
\label{chapter2}
\def\theequation{2.\arabic{equation}}
\setcounter{equation}{0}
\par
In this section, we analyze the pitchfork bifurcation, Hopf bifurcation and homoclinic orbit bifurcation of the system \eqref{1.2.3}.
Since when $\beta=\epsilon=0$, all points on the line $y=0$ are the equilibrium points of the system \eqref{1.2.3}, we assume $\beta$ and $\epsilon$ are not 0 at the same time in the sequel.

\begin{lemma}\label{L1}
The system \eqref{1.2.3} has only one equilibrium, $O$ in the case with $\beta=0$ or $\epsilon\leq0$; has three equilibrium points, $O$ and $E_{1,2}$ in the case with $\beta>0$ and $\epsilon>0$. Their properties are shown in the Table \ref{B2}, Where $\mu_{1}=\frac{\beta^2-\beta\epsilon-2\epsilon^2\sqrt{2\beta}}{\epsilon^2}$,  $\mu_{2}=\frac{\beta^2-\beta\epsilon+2\epsilon^2\sqrt{2\beta}}{\epsilon^2}$,
$\mu_{c}=\frac{-\epsilon\beta+\beta^2}{\epsilon^2}$.

\begin{table}[!ht]
\center
\caption{The type and stability of equilibrium points in different parameter value ranges}\label{B2}
\resizebox{\linewidth}{!}{
\begin{tabular}{|c|c|c|c|}
\hline
\multicolumn{3}{|c|}{possibilities of $(\epsilon, \beta, \mu)$} & type and stability\\ \hline
\multicolumn{3}{|c|}{$\epsilon\leq0,\beta>0,\mu\in R$} & $O(0,0)$, saddle  \\ \hline
\multicolumn{3}{|c|}{$\epsilon<0,\beta=0,\mu\in R$} & $O(0,0)$, saddle \\ \hline

\multirow{7}{*}{$\epsilon>0$} & \multirow{3}{*}{$\beta=0$} & $\mu<0$ & $O(0,0)$, stable node\\ \cline{3-4}
\multirow{7}{*}{} & \multirow{3}{*}{} & $\mu=0$ & $O(0,0)$, unstable focus\\ \cline{3-4}
\multirow{7}{*}{} & \multirow{3}{*}{} & $\mu>0$ & $O(0,0)$, unstable node\\ \cline{2-4}

\multirow{7}{*}{} & \multirow{3}{*}{$\beta>0$} & $\mu\leq\mu_{1}$ & $O(0,0)$, saddle; $E_{1}(-\sqrt{\frac{\beta}{\epsilon}},0)$, $E_{2}(\sqrt{\frac{\beta}{\epsilon}},0)$, stable nodes\\ \cline{3-4}
\multirow{7}{*}{} & \multirow{3}{*}{} & $\mu_{1}<\mu\leq\mu_{c}$ & $O(0,0)$, saddle; $E_{1}(-\sqrt{\frac{\beta}{\epsilon}},0)$, $E_{2}(\sqrt{\frac{\beta}{\epsilon}},0)$, stable foci\\ \cline{3-4}
\multirow{7}{*}{} & \multirow{3}{*}{} & $\mu_{c}<\mu<\mu_{2}$ & $O(0,0)$, saddle; $E_{1}(-\sqrt{\frac{\beta}{\epsilon}},0)$, $E_{2}(\sqrt{\frac{\beta}{\epsilon}},0)$, unstable foci\\ \cline{3-4}
\multirow{7}{*}{} & \multirow{3}{*}{} & $\mu\geq\mu_{2}$ & $O(0,0)$, saddle; $E_{1}(-\sqrt{\frac{\beta}{\epsilon}},0)$, $E_{2}(\sqrt{\frac{\beta}{\epsilon}},0)$, unstable nodes\\ \hline
\end{tabular}
}

\end{table}

\end{lemma}

\begin{proof}
(I) We first prove the case where the system\eqref {1.2.3} has only one equilibrium. When parameter $\beta = 0$ or $\beta > 0 $, $\epsilon\leq0 $, solving
\begin{align*}
  &y = 0, \\
  &(\mu+x^{2}-x^{4})y+\beta x-\epsilon x^{3} = 0,
\end{align*}
it is easy to see system\eqref {1.2.3} has a unique equilibrium $O(0, 0)$. Its corresponding Jacobian matrix $A_{0}$ for system\eqref {1.2.3} at $O(0, 0)$ is
\begin{equation*}
  A_{0}=\left(
  \begin{array}{cc}
     0 & 1 \\
     \beta & \mu
   \end{array}
  \right).
\end{equation*}
For the characteristic equation is,
$$\lambda^2-\mu\lambda-\beta=0,$$
its characteristic root can be solved as
$$\lambda_{1}=\frac{1}{2}(\mu+\sqrt{\mu^2+4\beta}),~~\lambda_{2}=\frac{1}{2}(\mu-\sqrt{\mu^2+4\beta}).$$
Since $\beta\geq0$, $\mu^2+4\beta\geq0$ is constant. Thus,
(1) when the $\beta > 0 $, $\epsilon \leq0 $, $\mu \in \mathbb{R} $, $\lambda_{1}, \lambda_{2} $ is real number and $\lambda_{1} > 0 $, $\lambda_{2} <0$. So the equilibrium $O(0,0)$ is the saddle of the system.
\par
(2) When the $\beta = 0 $, $\epsilon \neq0 $, $\mu \neq0 $, $\lambda_ {1}, \lambda_ {2}$ are real numbers and $\lambda_{1}=0 $,  $\lambda_ {2}=\mu $. So the equilibrium $O(0,0)$ is degenerate. We will analyze its stability using [\cite{L11}, Theorem 7.1 of Chapter 2], where the system is
\begin{align}\label{2.1.2}
    \left\{
        \begin{array}{ll}
        \dot{x}=y,\\
        \dot{y}=(\mu+x^{2}-x^{4})y-\epsilon x^{3}.
        \end{array}
    \right.
\end{align}
make $x\rightarrow x+\frac{y}{\mu}$, $y\rightarrow y$, $t\rightarrow\frac{t}{\mu}$ substitute into the system \eqref{2.1.2}, we have
\begin{align}\label{2.1.3}
    \left\{
        \begin{array}{ll}
        \dot{x}=&\frac{\epsilon}{\mu^2}x^3+\frac{3\epsilon-\mu}{\mu^3}x^2y+\frac{3\epsilon-\mu}{\mu^4}xy^2+\frac{\epsilon-\mu}{\mu^5}y^3\\
        &+\frac{1}{\mu^2}x^4y+\frac{4}{\mu^3}x^3y^2+\frac{6}{\mu^4}x^2y^3+\frac{4}{\mu^5}xy^4+\frac{1}{\mu^6}y^5\triangleq P(x,y),\\
        \dot{y}=&y-\frac{\epsilon}{\mu}x^3-\frac{3\epsilon-\mu}{\mu^2}x^2y-\frac{3\epsilon-\mu}{\mu^3}xy^2-\frac{\epsilon-\mu}{\mu^4}y^3\\
        &-\frac{1}{\mu}x^4y-\frac{4}{\mu^2}x^3y^2-\frac{6}{\mu3}x^2y^3-\frac{4}{\mu^4}xy^4-\frac{1}{\mu^5}y^5\triangleq y+Q(x,y).
        \end{array}
    \right.
\end{align}
Make $y + Q(x, y) = 0 $ obtained $y = \varphi (x) $, since $Q(x, y) $is analytic function, and $\varphi (0) = {\varphi} \dot (0) = 0 $, so the method of undetermined coefficients can be used for the $\varphi (x) $. Assuming $\varphi(x) = a_{2} x^2 + a_{3} x^3 + a_{4}x^4 +\cdots\cdots $, in $\varphi(x) + Q(x, \varphi(x) = 0$ and compare the coefficient get
\begin{align*}
 &a_{2}=0,~a_{3}-\frac{\epsilon}{\mu}=0,~a_{4}-\frac{-\mu+3\epsilon}{\mu^2}a_{2}=0,\\
 &a_{5}-\frac{-\mu+3\epsilon}{\mu^2}a_{3}-\frac{3\epsilon-2\mu}{\mu^3}a^2_{2}=0, \cdots
\end{align*}
solve for
\begin{align*}
  a_{2}=0,~ a_{3}=\frac{\epsilon}{\mu}, ~a_{4}=0, a_{5}=\frac{-\mu\epsilon+3\epsilon^2}{\mu^3}, \cdots,
\end{align*}
So we can get
\begin{align*}
  \varphi(x)=\frac{\epsilon}{\mu}x^3+\frac{-\mu\epsilon+3\epsilon^2}{\mu^3}x^5+\cdots.
\end{align*}
Substitute $y=\varphi(x)$ into $P(x,y)$ to get
\begin{align*}
  P(x,\varphi(x))=\frac{\epsilon}{\mu^2}x^3+\cdots,
\end{align*}
That is, when $m=3$, $a_{3}=\frac{\epsilon}{\mu^2}$.
Considering the time transformation $t=\frac{t}{\mu} $, when $\mu > 0 $, time transformation is positive, when $\mu<0 $, time transformation is the reverse. Hence, by [\cite{L11}, Theorem 7.1 of Chapter 2], we can get
when $\mu>0$, $\epsilon>0$, the equilibrium $O(0,0)$ is an unstable node of the system;
when $\mu>0$, $\epsilon<0$, the equilibrium $O(0,0)$ is a saddle of the system;
when $\mu<0$, $\epsilon>0$, the equilibrium $O(0,0)$ is a stable node of the system;
when $\mu<0$, $\epsilon>0$, the equilibrium $O(0,0)$ is a saddle of the system.

(3) When $\beta=0 $, $\epsilon\neq0 $, $\mu=0 $, two eigenvalues are equal to zero. In this case, the system is
\begin{align}\label{2.1.4}
    \left\{
        \begin{array}{ll}
        \dot{x}=y,\\
        \dot{y}=(x^{2}-x^{4})y-\epsilon x^{3}.
        \end{array}
    \right.
\end{align}
Obviously the corresponding [\cite{L11}, Theorem 7.2 of Chapter 2] the $P_{2}(x, y)\equiv0$, and $Q_{2}(x, y(x))=-\epsilon x^3$, $Q_{2y}(x,y(x))=x^2-x^4$, which in this case $m=1$, $n=2$, $a_{2m+1}=-\epsilon$.
So by [\cite{L11}, Theorem 7.2 of Chapter 2], when $\epsilon<0 $, $a_{2m+1} = -\epsilon>0 $, the equilibrium $O(0, 0)$ is a saddle; when $\epsilon>0$, $a_{2m+1}=-\epsilon<0$ and $n>m$, the equilibrium $O(0,0)$ is a center or focus. Let's use the Lyapnouv function method to determine the type of equilibrium $O(0,0)$ in the case of $\epsilon>0$.

For any $x\in(-1,1)$, construct the Lyapnouv function as
\begin{align}\label{2.1.4.1}
  V(x,y)=\frac{\epsilon}{4}x^4+\frac{1}{2}y^2.
\end{align}
Since $\epsilon>0$, $V(x,y)\geq0$ is a positive definite function, and
\begin{align}\label{2.1.5}
  \frac{dV}{dt}\mid_{(\ref{2.1.4})}=x^2y^2(1-x^2)\geq0.
\end{align}
Therefore, the equilibrium $O(0,0)$ is unstable. Also because
\begin{align*}
  \mathrm{div}\{y,(x^{2}-x^{4})y-\epsilon x^{3}\}=x^2-x^4,
\end{align*}
at $-1<x<1$ is invariant and not constant zero, so there is no closed orbit. So in this case $O(0,0)$ is an unstable focus.

(II) When $\epsilon>0$, $\beta>0$, we know from
\begin{align*}
  &y=0,\\
  &(\mu+x^{2}-x^{4})y+\beta x-\epsilon x^{3}=0,
\end{align*}
that the system has three equilibria $O(0,0)$, $E_{1}(x_{1},0)$ and $E_{2}(x_{2},0)$, where $x_{1}=-\sqrt{\frac{\beta}{\epsilon}}$, $x_{2}=\sqrt{\frac{\beta}{\epsilon}}$.
According to the previous analysis, the origin $O(0,0)$ is a saddle of the system\eqref{1.2.3}. Let's analyze the stability of the other two equilibria, $E_{1}(x_{1},0)$ and $E_{2}(x_{2},0)$. $E_{1}(x_{1},0)$ and $E_{2}(x_{2}, 0)$ corresponding Jacobian matrix, the same as
\begin{equation*}
  B=\left(
     \begin{array}{cc}
       0& 1 \\
       -2\beta & \mu+\frac{\beta}{\epsilon}-\frac{\beta^2}{\epsilon^2} \\
     \end{array}
   \right).
\end{equation*}
By Jacobian matrix to get the equilibrium $E_{1}(x_{1},0)$ and $E_{2}(x_{2},0)$ corresponding characteristic equation\eqref{2.1.7}
\begin{equation}\label{2.1.7}
  \lambda^2-(\mu+\frac{\beta}{\epsilon}-\frac{\beta^2}{\epsilon^2})\lambda+2\beta=0.
\end{equation}
The eigenvalue can be calculated as
$$\lambda_{1}=\frac{\epsilon^2\mu+\epsilon\beta-\beta^2}{2\epsilon^2}+\frac{\sqrt{\Delta}}{2},$$
$$\lambda_{2}=\frac{\epsilon^2\mu+\epsilon\beta-\beta^2}{2\epsilon^2}-\frac{\sqrt{\Delta}}{2},$$
where $\Delta=(\mu+\frac{\beta}{\epsilon}-\frac{\beta^2}{\epsilon^2})^2-8\beta.$ For simplicity, let's say
$\mu_{c}=\frac{-\epsilon\beta+\beta^2}{\epsilon^2}$,
\begin{align*}
  \mu_{1}&=\frac{\beta^2-\beta\epsilon-2\epsilon^2\sqrt{2\beta}}{\epsilon^2}=\mu_{c}-2\sqrt{2\beta}, \\
  \mu_{2}&=\frac{\beta^2-\beta\epsilon+2\epsilon^2\sqrt{2\beta}}{\epsilon^2}=\mu_{c}+2\sqrt{2\beta}.
\end{align*}
Through simple calculations we can see:
when $\mu\leq\mu_{1}$, be able to get $\Delta\geq0$, and the characteristic values are negative real numbers, so $E_{1}(x_{1},0)$ and $E_{2}(x_{2},0)$ are stable nodes;
when $\mu\geq\mu_{2} $, be able to get $\Delta\geq0$, and the characteristic values are positive real numbers, so $E_{1}(x_{1}, 0) $ and $E_{2}(x_{2},0)$ are unstable nodes;
when $\mu_{1}<\mu<\mu_{c}$, you can get $\Delta<0$, imaginary eigenvalues for a negative real component, so $E_{1}(x_{1},0)$ and $E_{2}(x_{2},0)$ are stable foci;
when $\mu_{c}<\mu<\mu_{2}$, you can get $\Delta<0 $, imaginary eigenvalues for positive real part, so $E_{1}(x_{1},0)$ and $E_{2}(x_{2},0)$ are unstable foci.
\end{proof}
\par

Based on Lemma \ref{L1}, we investigate bifurcations from finite equilibria in the following two propositions.
\begin{proposition}\label{p1}
When $\beta$ passes through $0$, the autonomous system\eqref{1.2.3} will generate two pitchfork bifurcations in the small neighborhood of the equilibrium $O(0,0)$: \\
(i) If $\epsilon>0,\mu>0$, a subcritical pitchfork bifurcation is generated at $\beta=0$. \\
(ii) If $\epsilon>0,\mu<0$, a supercritical pitchfork bifurcation is generated at $\beta=0$.
\end{proposition}

\begin{proof}
By restricting the discussion of autonomous system\eqref{1.2.3} to the central manifold, taking $\beta$ as the bifurcation parameter and $\beta$ as the new independent variable, system\eqref{1.2.3} is rewritten as
\begin{equation}\label{2.2.1.1}
  \left\{
  \begin{array}{l}
          \dot{x}=y,\\
          \dot{y}=\mu y+\beta x-\epsilon x^3+(x^2-x^4)y-\epsilon x^3,\\
          \dot{\beta}=0,
        \end{array}
  \right.
\end{equation}
where $\beta x$ is a nonlinear term.
Obviously $(x,\beta,y)=(0,0,0)$ is a fixed point with eigenvalues of $0$ and $\mu$. The corresponding eigenvectors are
\begin{align*}
  v_{1}=(1,0)^{T} , ~v_{2}=(1,\mu)^{T},
\end{align*}
by invertible linear transformation
$$X\rightarrow TX,$$
where $X=(x,y)^{T}$,~~$T=(v_{1},v_{2}),$
system\eqref{2.2.1.1} into
\begin{equation}\label{2.2.1.2}
 \begin{array}{l}
 \left(
   \begin{array}{c}
     \dot{x} \\
     \dot{y} \\
   \end{array}
 \right)=
 \left(
   \begin{array}{cc}
     0 & 0 \\
     0 & \mu \\
   \end{array}
 \right)
 \left(
   \begin{array}{c}
     x \\
     y \\
   \end{array}
 \right)+
 \left(
   \begin{array}{c}
     -\frac{\beta}{\mu} x+\frac{\epsilon}{\mu}x^3+yF_{1}(x,y), \\
     \frac{\beta}{\mu} x-\frac{\epsilon}{\mu}x^3+yF_{2}(x,y) \\
   \end{array}
 \right),\\

   ~~~\dot{\beta}~~=~~0,
 \end{array}
\end{equation}
where $F_{1}(x,y),~F_{2}(x,y)$ are function of $x,~y$.
For sufficiently small $x$,~$\beta$ has
\begin{align*}
  W^{c}(0)=\{(x,y,\beta)\in\mathbb{R}^{3}:y=h(x,\beta),h(0,0)=0,Dh(0,0)=0\}.
\end{align*}
Can calculate the center manifold is $y=h(x,\beta)=0$.

So system\eqref{2.2.1.2} restrictions on the center manifold is
\begin{align}\label{2.2.1.4}
  \begin{split}
    \dot{x}&=-\frac{\beta}{\mu}x+\frac{\epsilon}{\mu}x^3, \\
    \dot{\beta}&=0.\
  \end{split}
\end{align}
Make $\bar{\beta}=-\frac{\beta}{\mu}$, then the system\eqref{2.2.1.4} to write to
\begin{align*}
  \begin{split}
    \dot{x}&=\bar{\beta}x+\frac{\epsilon}{\mu}x^3, \\
    \dot{\bar{\beta}}&=0.\
  \end{split}
\end{align*}
Since $\beta\geq0$, $\epsilon>0$, when $\mu>0$, there is $\frac{\epsilon}{\mu}>0$, then the system\eqref{1.2.3} has a subcritical pitchfork bifurcation; when $\mu<0$, $\frac{\epsilon}{\mu}<0$, this system \eqref{1.2.3} there is supercritical pitchfork bifurcation.

\end{proof}

\begin{proposition}\label{p2}
If $\epsilon>0$, $\beta>0$ and $\mu_{1}<\mu<\mu_{2}$, the equilibrium $E_{1}$ (or $E_{2}$) is a focus of the system\eqref{1.2.3}. When $\mu>\mu_{c}$ the system\eqref{1.2.3} yields a stable limit cycle in a small area of the equilibrium $E_{1}$(or $E_{2}$).
\end{proposition}

\begin{proof}
Since Norm-form is required later in the analysis of the existence of quasi-periodic solutions for periodic forced systems, here we derive the standard form of Hopf bifurcation.
From Lemma\eqref{L1} , we can see that in the parameter plane $$H=\{(\mu,\beta):\mu=\frac{-\epsilon\beta+\beta^2}{\epsilon^2},\epsilon>0,\beta>0\}$$. The equilibrium of the system\eqref{1.2.3} $E_{1}$ (or $E_{2}$) has a pair of pure virtual eigenvalues, $\lambda_{1,2}=\pm{\rm i}\sqrt{2\beta}$. Since the equilibrium $E_{1}$ and $E_{2}$ are symmetric about the axis of $y$, only one equilibrium $E_{2}$ can be discussed.

First through coordinate transformation
\begin{equation}\label{2.2.3}
  \left\{
  \begin{array}{l}
          x\rightarrow\sqrt{\frac{\beta}{\epsilon}}+x,\\
           y\rightarrow y,
        \end{array}
  \right.
\end{equation}
equilibrium $E_{2}$ move to the origin, get it
\begin{equation}\label{2.2.4}
\left\{
\begin{array}{l}
   \dot{x}=y; \\
  \dot{y}=-2\beta x+(\mu+\frac{\beta}{\epsilon}
  -\frac{\beta^2}{\epsilon^2})y-3\epsilon\sqrt{\frac{\beta}{\epsilon}}x^2+2\sqrt{\frac{\beta}{\epsilon}}(1-2\frac{\beta}{\epsilon})xy\\
        \qquad      -\epsilon x^3+(1-6\frac{\beta}{\epsilon})x^2y-4\sqrt{\frac{\beta}{\epsilon}}x^3y-x^4y.
\end{array}
\right.
\end{equation}
Make
$$\sigma(\mu)=\mu+\frac{\beta}{\epsilon}-\frac{\beta^2}{\epsilon^2},~\zeta(\mu)=2\beta,~\sigma(\mu_{c})=0,~\zeta(\mu_{c})=\omega^2_{0}=2\beta.$$
Let $\lambda(\mu)=\delta(\mu)+{\rm i}\omega(\mu)$ be the eigenvalue of $E_{2}$ and $\bar{\lambda}(\mu)$ be the other eigenvalue. For small $|\mu|$, we have
$$\delta(\mu)=\frac{1}{2}\sigma(\mu),~\omega(\mu)=\frac{1}{2}\sqrt{4\zeta(\mu)-\sigma^2(\mu)},$$
and $\delta(\mu_{c})=0, ~\omega(\mu_{c})=\omega_{0}=\sqrt{2\beta}>0.$
Eigenvalue $\lambda, ~\bar{\lambda}$ corresponding eigenvector is $v_{1}, ~v_{2}$, where
$$v_{1}=(\frac{\sigma(\mu)-{\rm i}\sqrt{4\zeta(\mu)-\sigma^2(\mu)}}{4\beta},1)^{T}=(\frac{\bar{\lambda}}{2\beta},1)^{T},$$
$$v_{2}=(\frac{\sigma(\mu)+{\rm i}\sqrt{4\zeta(\mu)-\sigma^2(\mu)}}{4\beta},1)^{T}=(\frac{\lambda}{2\beta},1)^{T}.$$
Take the non-degenerate matrix
\begin{equation*}
  T=(v_{1},v_{2})^{T},
\end{equation*}
and take a linear transformation
\begin{equation}\label{2.2.5}
  X=TZ,
\end{equation}
where $X=(x,y)^{T}$, $Z=(z,\bar{z})^{T}$ and $z$ is conjugated with $\bar{z}$. Then the system\eqref {2.2.4} can be turned into
\begin{equation}\label{2.2.12}
  \left(
  \begin{array}{cc}
     \dot{z}  \\
     \dot{\bar{z}}
   \end{array}
   \right)
   = \left(
  \begin{array}{cc}
     \lambda & 0  \\
     0 & \bar{\lambda}
   \end{array}
   \right)
   \left(
  \begin{array}{cc}
     z  \\
     \bar{z}
   \end{array}
   \right)+
   \left(
  \begin{array}{cc}
     N_{1}(Z)  \\
     N_{2}(Z)
   \end{array}
   \right),
\end{equation}
where
\begin{align*}
  N_{1}(Z)=\sum\limits_{2\leq k+l\leq3}g_{kl}z^{k}\bar{z}^{l}+O(|z|^4),  \\
  N_{2}(Z)=\sum\limits_{2\leq k+l\leq3}\bar{g}_{kl}z^{l}\bar{z}^{k}+O(|z|^4),
\end{align*}
\begin{align*}
  &g_{20}=-\frac{(\frac{\beta}{\epsilon})^{\frac{3}{2}}\lambda\bar{\lambda}(8\beta^2-4\beta\epsilon+3\epsilon^2\bar{\lambda})}{4\beta^3(\lambda-\bar{\lambda})},\\
  &g_{11}=-\frac{(\frac{\beta}{\epsilon})^{\frac{3}{2}}\lambda(3\epsilon^2\lambda\bar{\lambda}+4\beta^2(\lambda+\bar{\lambda})-2\beta\epsilon(\lambda+\bar{\lambda}))}{2\beta^3(\lambda-\bar{\lambda})},\\
  &g_{02}=-\frac{(\frac{\beta}{\epsilon})^{\frac{3}{2}}\lambda^2(8\beta^2-4\beta\epsilon+3\epsilon^2\lambda)}{4\beta^3(\lambda-\bar{\lambda})},\\
  &g_{30}=-\frac{\lambda\bar{\lambda}^2(12\beta^2-2\beta\epsilon+\epsilon^2\bar{\lambda})}{8\beta^3\epsilon(\lambda-\bar{\lambda})},\\
  &g_{21}=-\frac{\lambda\bar{\lambda}(3\epsilon^2\lambda\bar{\lambda}+12\beta6(2\lambda+\bar{\lambda})-2\beta\epsilon(\bar{\lambda}+2\lambda))}{8\beta^3\epsilon(\lambda-\bar{\lambda})},\\
  &g_{12}=-\frac{\lambda^2(3\epsilon^2\lambda\bar{\lambda}+12\beta6(\lambda+2\bar{\lambda})-2\beta\epsilon(2\bar{\lambda}+\lambda))}{8\beta^3\epsilon(\lambda-\bar{\lambda})},\\
  &g_{03}=-\frac{\lambda^3(12\beta^2-2\beta\epsilon+\epsilon^2\lambda)}{8\beta^3\epsilon(\lambda-\bar{\lambda})},
\end{align*}
$\bar{g}_{kl}$ represents the complex conjugation of $g_{kl}$.

Because system\eqref{2.2.12} in the first equation and the second equation is conjugate relation, so we just need to discuss the first equation, namely
\begin{align}\label{2.2.6}
  \dot{z}=\lambda z+\sum\limits_{2\leq k+l\leq3}g_{kl}z^{k}\bar{z}^{l}+O(|z|^4),
\end{align}

Next, the Hopf bifurcation of the system is analyzed by the method of finding the Nomal-Form, and the quadratic term in system\eqref{2.2.6} is eliminated by the transformation
\begin{align}\label{2.2.7}
  z\rightarrow z+\frac{h_{20}}{2}z^2+h_{11}z\bar{z}+\frac{h_{02}}{2}\bar{z}^2,
\end{align}
where $h_{20}=\frac{g_{20}}{\lambda}$, $h_{11}=\frac{g_{11}}{\bar{\lambda}}$, $h_{02}=\frac{g_{02}}{2\bar{\lambda}-\lambda}$.
The transformed equation is
\begin{equation*}
  \dot{z}=\lambda z+\frac{k_{30}}{6}z^3+\frac{k_{21}}{2}z^2\bar{z}+\frac{k_{12}}{2}z\bar{z}^2+\frac{k_{03}}{6}\bar{z}^3+\cdots\cdots
\end{equation*}
where
\begin{align*}
 &k_{30}=g_{30}+\frac{3g^2_{20}}{\lambda}+\frac{3g_{11}\bar{g_{02}}}{2\lambda-\bar{\lambda}},\\
 &k_{21}=g_{21}+\frac{g_{02}\bar{g_{02}}}{2\lambda-\bar{\lambda}}+\frac{2g_{11}\bar{g_{11}}}{\lambda}+\frac{g_{11}g_{20}(2\lambda+\bar{\lambda})}{\lambda\bar{\lambda}},\\
 &k_{12}=g_{12}+\frac{2g_{02}\bar{g_{11}}}{\lambda}+\frac{2g_{11}^2}{\bar{\lambda}}+\frac{g_{11}\bar{g_{20}}}{\bar{\lambda}}-\frac{g_{02}g_{20}}{\lambda-2\bar{\lambda}},\\
 &k_{03}=g_{03}+\frac{3g_{02}\bar{g_{20}}}{\bar{\lambda}}-\frac{3g_{11}g_{02}}{\lambda-2\bar{\lambda}},
\end{align*}
a further transformation is made to eliminate the cubic off-resonance term, and the transformation is given by
\begin{align}\label{2.2.8}
  z\rightarrow z+\frac{h_{30}}{6}z^3+\frac{h_{12}}{6}z\bar{z}^2+\frac{h_{03}}{2}\bar{z}^3,
\end{align}
where $h_{30}=\frac{g_{30}}{2\lambda}$, $h_{12}=\frac{g_{12}}{2\bar{\lambda}}$, $h_{03}=\frac{g_{03}}{3\bar{\lambda}-\lambda}$.
Substitute it into the system, we have
\begin{align}\label{2.2.9}
  \dot{z}=(\delta(\mu)+{\rm i}\omega(\mu))z+c_{1}z^2\bar{z}+O(|z|^4),
\end{align}
where $c_{1}$ is the function related to the parameter $\mu$,
\begin{align*}
  c_{1}=\frac{g_{21}}{2}+\frac{g_{02}\bar{g}_{02}}{2(2\lambda-\bar{\lambda})} +\frac{g_{11}\bar{g}_{11}}{\lambda}+\frac{g_{11}g_{20}(2\lambda+\bar{\lambda})}{2\lambda\bar{\lambda}}.
\end{align*}
So it can be calculated at $\mu=\mu_{c}$
\begin{align*}
  c_{1}(\mu_{c})=-\frac{1}{2\beta}<0.
\end{align*}
Also because $\delta(\mu)=\frac{1}{2}(\mu+\frac{\beta}{\epsilon}-\frac{\beta^2}{\epsilon^2})$, so
\begin{align*}
  \frac{d\delta(\mu)}{d\mu}\mid_{\mu=\mu_{c}}=\frac{1}{2}>0.
\end{align*}
This shows that both the nondegeneracy and transversality conditions are satisfied, so there is a coordinate transformation and a time transformation
\begin{align*}
  z=\frac{z}{\sqrt{|l_{1}(\alpha)|}}, ~dt=(1+e_{1}(\alpha)|z|^2)\omega(\mu)dt,
\end{align*}
where
\begin{align*}
  \alpha=\alpha(\mu)=\frac{\delta(\mu)}{\omega(\mu)}, ~ e_{1}(\alpha)=\Im\frac{c_{1}(\mu(\alpha))}{\omega(\mu(\alpha))}, ~ l_{1}(\alpha)=\Re\frac{c_{1}(\mu(\alpha))}{\omega(\mu(\alpha))}-\alpha ~ e_{1}(\alpha),
\end{align*}
and
\begin{align*}
  l_{1}(0)=\frac{\Re c_{1}(\mu_{c})}{\omega(\mu_{c})}.
\end{align*}
Normalized the system\eqref{2.2.9}
\begin{align}\label{2.2.10}
  \dot{z}=(\alpha+{\rm i})z-z^2\bar{z}+O(|z|^4).
\end{align}
Substitute $z=\rho e^{{\rm i}\varphi}$ into the equation\eqref{2.2.10} and separate the real and imaginary parts to get the system in polar coordinates
\begin{equation}\label{2.2.11}
  \left\{\begin{array}{l}
             \dot{\rho}=\rho(\alpha-\rho^2), \\
             \dot{\varphi}=1.
        \end{array}
  \right.
\end{equation}

Can be seen when $\alpha\leq0 $, or $\mu\leq\frac{\beta^2}{\epsilon^2}-\frac {\beta}{\epsilon}$, the system\eqref{2.2.11} is only an equilibrium solution $\rho=0 $. When $\alpha>0$, that is, $\mu>\frac{\beta^2}{\epsilon^2}-\frac{\beta}{\epsilon}$, the system has two equilibrium solutions: $\rho=0$ and $\rho=\sqrt{\alpha}$. However, $\rho=\sqrt{\alpha}$ is a limit cycle, and since $c_{1}(\mu_{c})<0$, the first Lyapunov coefficient $l_{1}<0$, it can be known that the limit cycle is stable by Hopf bifurcation theory\cite{L9}.
\end{proof}

\begin{lemma}\label{L3}
If $\epsilon>0$, $\beta>0$ and $\mu=\mu_{3}$, there are three equilibria in the system\eqref{1.2.3}, where the origin $O(0,0)$ is a saddle, and $E_{1}$ , $E_{2}$ are unstable foci. In this case, there are two homoclinic orbits that co-reside at the saddle.
\end{lemma}

\begin{proof}
In order to more conveniently study the homologous orbital bifurcation, we introduce the following transformation:
\begin{align}\label{2.2.3.1}
  x\rightarrow\epsilon_{1}x,~y\rightarrow\epsilon^2_{1}y,~\beta\rightarrow\epsilon^2_{1}\beta,~  \mu\rightarrow\epsilon^2_{1}\mu,~t\rightarrow\frac{1}{\epsilon_{1}}t.
\end{align}
The system\eqref{1.2.3} can be converted to
\begin{equation}\label{2.2.3.2}
\left\{\begin{array}{l}
             \dot{x}=y, \\
             \dot{y}=\beta x-\epsilon x^3+\epsilon_{1}(\mu y+x^2y-\epsilon^2_{1}x^4y).
        \end{array}
  \right.
\end{equation}
When $\epsilon_{1}=0$, the system\eqref{2.2.3.2} is
\begin{equation*}
  \left\{\begin{array}{l}
             \dot{x}=y, \\
             \dot{y}=\beta x-\epsilon x^3,
        \end{array}
  \right.
\end{equation*}
this is a Hamilton system. Its energy function is
\begin{align*}
  H(x,y)=\frac{1}{2}y^2-\frac{\beta}{2}x^2+\frac{\epsilon}{4}x^4=h.
\end{align*}
when $h=0$, there is a pair of homologous tracks, denoted as $\Gamma^{0}=\{\Gamma^{0}_{-}(t)|t\in R\}\bigcup\{O\}\bigcup\{\Gamma^{0}_{+}(t)|t\in R\}$, It consists of a hyperbolic saddle $O(0,0)$ and two homoclinic orbital lines homoclinic at the saddle $\{\Gamma^{0}_{-}(t)|t\in R\}$, $\{\Gamma^{0}_{+}(t)|t\in R\}$. Make
\begin{align*}
  x(t) =& l_{1}\mathrm{sech}(t), \\
  y(t) =& l_{2}\mathrm{sech}(t)\tanh(t),
\end{align*}
by substituting the energy function equation $H(x,y)=0$, when $h=0$, the equation for the homoclinic orbits can be written as
\begin{equation*}
  \left\{\begin{array}{l}
             x(t)=\pm\sqrt{\frac{2\beta}{\epsilon}}\mathrm{sech}(t), \\
             y(t)=\mp\frac{\sqrt{2}\beta}{\sqrt{\epsilon}}\mathrm{sech}(t)\tanh(t).
        \end{array}
  \right.
\end{equation*}
Since the perturbation term of the system\eqref{2.2.3.2} is $\epsilon_{1}(\mu y+x^2y-\epsilon^2_{1}x^4y)$, and depends on the time $t$. So the Melnikov function can be written as:
\begin{align}\label{2.2.3.6}
  M(\beta, \epsilon, \mu)= & \int_{-\infty}^{\infty}y(t)(\mu y(t)+x^2(t)y(t)-\epsilon^2_{1}x^4(t)y(t))dt\notag \\
                         = & \int_{-\infty}^{\infty}y(t)^2(\mu +x^2(t)-\epsilon^2_{1}x^4(t))dt\notag\\
                         = & \int_{-\infty}^{\infty}\frac{2\beta^2}{\epsilon}\mathrm{sech}^2(t)\tanh^2(t)(\mu+\frac{2\beta}{\epsilon}\mathrm{sech}^2(2)- \frac{4\beta^2\epsilon^2_{1}}{\epsilon^2}\mathrm{sech}^4(t))dt\notag \\
                         = & \frac{4\beta^2\mu}{3\epsilon}+\frac{16\beta^3}{15\epsilon^2}-\frac{128\beta^4\epsilon^2_{1}}{105\epsilon^3}.
\end{align}
In order for the homoclinic orbits to survive in the disturbance, this requires $M(\beta, \epsilon, \mu)=0$. Therefore, from \eqref{2.2.3.6}, the bifurcation curve of the homoclinic orbits can be obtained:
\begin{align}\label{2.2.3.7}
  \mu=\frac{32\beta^2\epsilon^2_{1}}{35\epsilon^2}-\frac{4\beta}{5\epsilon}.
\end{align}
Then substituting the inverse transformation of the parameter transformation in \eqref{2.2.3.1} $\beta\rightarrow\frac{\beta}{\epsilon_{1}^2}$, $\mu\rightarrow\frac{\mu}{\epsilon_{1}^2}$ into \eqref{2.2.3.7}, the homoclinic orbit bifurcation curve is obtained as follows:
\begin{align*}
  \mu=\frac{32\beta}{35\epsilon^2}-\frac{4\beta}{5\epsilon}\triangleq\mu_{3},
\end{align*}
That is, when $\mu=\mu_{3}$, the system\eqref{1.2.3} generated a homoclinic orbit.
\end{proof}

In this way, we have obtained all the bifurcations of the autonomous system \eqref{1.2.3} and the corresponding bifurcation curves. In order to prove the Theorem \ref{T2.2.0}, it is necessary to discuss the system\eqref{1.2.3} when the parameter lies in the range. The problem of large limit rings exists when \normalsize{\textcircled{\scriptsize{1}}}. Before analyzing large limit cycles, it is necessary to obtain the stability of the equilibrium at infinity of the system\eqref{1.2.3}.

\section{Limit cycles and homoclinic loops}
\label{chapter3}
\def\theequation{3.\arabic{equation}}
\setcounter{equation}{0}
\par
In this section, we will analyze the behavior of the system \eqref{1.2.3} at infinity and the existence and non-existence of large limit cycles, and give a representative global phase diagram. First, for equilibria at infinity of the unforced system\eqref{1.2.3}, we conclude as follows.
\begin{lemma}\label{L4}
System\eqref{1.2.3} in Poincar\'{e} the disc has four infinity equilibria $\bar {B} (1, 0) $, $\bar {\bar {B}} (1, 0) $, $\bar (0, 1) ${C}, $\bar {\bar {C}} (0, 1) $, where $\bar {C} $ and $\bar {\bar {C}}$ are unstable nodes; when $\epsilon \leq0 $, $\bar {B} $ and $\bar {\bar {B}}$ are stable nodes; when $\epsilon>0$, $\bar{B}$ and $\bar{\bar{B}}$ are saddle.
\end{lemma}
\begin{proof}
In order to analyze equilibria at infinity for the unforced system\eqref{1.2.3}, it is first necessary to map infinity onto the Poincar\'{e} disk. By the Poincar\'{e} transformation
\begin{align*}
x=\frac{1}{z},~~y=\frac{u}{z},
\end{align*}
mapping points on the equator onto the plane $(u,z)$, we get
\begin{align*}
  \dot{x}=-\frac{1}{z^2}\dot{z},~\dot{y}=\frac{\dot{u}z-\dot{z}u}{z^2},
\end{align*}
substituting into the system\eqref{1.2.3}, we can get
\begin{equation}\label{2.3.1.1}
 \left\{
 \begin{array}{ll}
  \dot{u}=-u^{2}+(\mu z^{2}+1-\frac{1}{z^{2}})\frac{u}{z^{2}}+\beta-\frac{\epsilon}{z^{2}},\\
  \dot{z}=-zu.
 \end{array}
 \right.
\end{equation}
The space of the variable $z$ at the position of the denominator, when $z=0$ is meaningless, so we need to travel through the time transformation $dt=z^{4}d\tau $, get an equivalent system
\begin{equation}\label{2.3.1.2}
 \left\{
 \begin{array}{ll}
  \frac{du}{d\tau}=-u-\epsilon z^{2}+z^{2}u+\beta z^{4}+\mu z^{4}u-u^{2}z^{4},\\
  \frac{dz}{d\tau}=-z^{5}u.
 \end{array}
 \right.
\end{equation}
Obviously the system\eqref{2.3.1.2} only has one equilibrium $B(0,0)$, which corresponds to the Jacobian matrix is
\begin{equation*}
  \left(
     \begin{array}{cc}
       -1 & 0 \\
       0 & 0 \\
     \end{array}
   \right),
\end{equation*}
this is a degenerate equilibrium with eigenvalues $\lambda_{1}=-1$ and $ \lambda_{2}=0.$

In order to facilitate the analysis, the variable in the equation\eqref {2.3.1.2} is still represented by $x,y$ through the time transformation $d\tau\rightarrow-d\tau$ to get
\begin{equation*}\label{2.3.1.3}
  \left\{
  \begin{array}{l}
     \frac{dx}{d\tau}=x^5y\triangleq\varphi(x,y), \\
     \frac{dy}{d\tau}=y+\epsilon x^2-x^2y-\beta x^4-\mu x^4y+x^4y^2\triangleq y+\psi(x,y).
   \end{array}
   \right.
\end{equation*}
Implicit function $y=\phi(x)$ can be solved form $y+\psi(x,y)=0$ by the implicit function theorem. And because $\psi(x,y)$ is an analytic function, and  $\phi(0)=\dot{\phi}(0)=0$, so the method of undetermined coefficients can be used for the $\phi(x)$, we can assume
\begin{align*}
  \phi(x)=a_{2}x^2+a_{3}x^3+a_{4}x^4+\cdots.
\end{align*}
Substituting $\phi+\psi(x,\phi)=0$, compare the coefficients to determine the implicit function
$$\phi(x)=-\epsilon x^2+(\beta-\epsilon)x^4+\cdots.$$
Put in $\varphi(x,y)$
\begin{equation*}
  \varphi(x,\phi(x))=-\epsilon x^7+(\beta-\epsilon)x^9+\cdots,
\end{equation*}
when$\epsilon\neq0$, $m=7,~a_{7}=-\epsilon,$
when$\epsilon=0$, $m=9, ~a_{9}=\beta>0.$

Therefore, according to [\cite{L11}, Theorem 7.1 of Chapter 2], when $\epsilon\leq0$, $B(0,0)$ is an unstable node; when $\epsilon>0$, $B(0,0)$ is a saddle. Notice that we did the time transform $d\tau\rightarrow-d\tau$, so we need to return to the original time $\tau$, that is, when $\epsilon\leq0$, $B(0,0)$ is a stable node; when $\epsilon>0$, $B(0,0)$ is a saddle.

Since the time transformation $dt=z^{4}d\tau$ has an even number of degrees,
therefore, the corresponding point of $B(0,0)$ on the Poincar\'{e} disk $\bar{B}(1,0)$, and $\bar{\bar{B}}(-1,0)$ is consistent with the stability and trajectory trend of $B(0,0)$.

To make $x=\frac{v}{z},~~y=\frac{1}{z}$, Plug into the system\eqref{1.2.3} to get
\begin{equation*}\label{2.3.1.4}
 \left\{
 \begin{array}{ll}
  \dot{v}=1-\beta v^{2}-(\mu+\frac{v^{2}}{z^{2}}-\frac{v^{4}}{z^{4}})v+\epsilon\frac{v^{4}}{z^{2}},\\
  \dot{z}=-\mu z-\beta vz-\frac{v^{2}}{z}+\epsilon\frac{v^{3}}{z}+\frac{v^{4}}{z^{3}}.
 \end{array}
 \right.
\end{equation*}
Similarly, you need to do the time transformation $dt=z^{4}d\tau$, which is
\begin{equation}\label{2.3.1.5}
 \left\{
 \begin{array}{ll}
  \frac{dv}{d\tau}=z^{4}+v^{5}-\mu z^{4}v-\beta v^{2}z^{4}-v^{3}z^{2}+\epsilon v^{4}z^{2},\\
  \frac{dv}{d\tau}=-\mu z^{5}-v^{2}z^{3}+v^{4}z-\beta vz^{5}+\epsilon v^{3}z^{3}.
 \end{array}
 \right.
\end{equation}
Obviously, $C(0,0)$ is a equilibrium of the system \eqref{2.3.1.5}, and its corresponding Jacabian matrix is the zero matrix, and both eigenvalues are 0. Since the equilibrium is degenerate, the type of equilibrium can not be directly judged, so it is necessary to discuss the characteristic direction and trajectory direction. By the polar coordinate transformation $v=r\cos\theta$,~$z=r\sin\theta$, \eqref{2.3.1.5} is reduced to a system in polar coordinates
\begin{equation*}\label{2.3.1.6}
  \left\{
  \begin{array}{ll}
    \dot{r}=r^4\cos\theta \sin^4\theta+O(r^5),\\
    r\dot{\theta}=-r^4\sin^5\theta.
  \end{array}
  \right.
\end{equation*}
That is
\begin{align*}\label{2.3.1.7}
 \frac{1}{r}\frac{dr}{d\theta}=\frac{\cos\theta \sin^4\theta+O(r)}{-\sin^5\theta}=\frac{H(\theta)+O(1)}{G(\theta)+O(1)}.
\end{align*}
So the characteristic equation $G(\theta)=-\sin^5\theta$, the value that satisfies $G(\theta)=0$ has $0,\pi$. While $H(\theta)=\cos\theta\sin^4\theta$, and $H(0)=H(\pi)=0 $. It cannot be directly determined whether $\theta=0$ and $\theta=\pi$ are one characteristic direction.

Below Briot-Bouquet transformation in rail line, the $ v = v$,  $z = z_ {1} v $, $dt_{1} = vd \tau $, system \eqref {2.3.1.5} as
\begin{equation}\label{2.3.1.8}
  \left\{
  \begin{array}{l}
    \frac{dv}{dt_{1}}=v^4-v^4z^2_{1}+v^3z^4_{1}-\mu v^4z^4_{1}+\epsilon v^5z^2_{1}-\beta v^5z^4_{1}, \\
    \frac{dz_{1}}{dt_{1}}=v^2z^5_{1}.
  \end{array}
  \right.
\end{equation}
At this point, the origin is still a degenerate equilibrium. By the polar coordinate transformation $v=r\cos\theta$, $z_{1}=r\sin\theta$, get
\begin{align*}
  \frac{dr}{rd\theta}=\frac{\cos^5\theta+O(r^1)}{-\cos^4\theta \sin\theta+O(r^1)}.
\end{align*}
So $G(\theta)=-\cos^4\theta \sin\theta$, $H(\theta)=\cos^5\theta$, such that the solution of $G(\theta)=0$ is $\theta_{0}=0,\pi,\frac{\pi}{2},\frac{3\pi}{2}$.

Because of $G^{'}(0)H(0)=G^{'}(\pi)H(\pi)=-1<0$, therefore, $t_{1}\rightarrow\infty$ has a unique path along $\theta=0$, $\theta=\pi$ to enter the origin. Because $H(\frac{\pi}{2})=H(\frac{3\pi}{2})=0$, so not sure whether the two direction rail line into the origin. Repeat the above steps to make $v=v$, $z_{1}=z_{2}v$, $dt_{2}=vdt_{1}$ into system\eqref{2.3.1.8}
\begin{equation}\label{2.3.1.9}
  \left\{
  \begin{array}{l}
    \frac{dv}{dt_{2}}=v^3-v^5z^2_{2}+\epsilon v^6z^2_{2}+v^6z^4_{2}-\mu v^7z^4_{2}-\beta v^8z^4_{2}, \\
    \frac{dz_{2}}{dt_{2}}=-v^2z_{2}+v^4z^3_{2}-\epsilon v^5z^3_{2}+\mu v^6z^5_{2}+\beta v^7z^5_{2}.
  \end{array}
  \right.
\end{equation}
By polar coordinate transformation $v=r\cos\theta$, $z_{2}=r\sin\theta$, get $G(\theta)=-2\cos^3\theta \sin\theta$, $H(\theta)=\cos^2\theta(\cos^2\theta-\sin^2\theta)$, such that the solution of $G(\theta)=0$ is  $\theta_{0}=0,\pi,\frac{\pi}{2},\frac{3\pi}{2}$. Due to $G^{'}(0)H(0)=G^{'}(\pi)H(\pi)=-2<0$, therefore, when $t_{2}\rightarrow\infty$, there is a unique path along $\theta=0$, $\theta=\pi$ to enter the origin. Because $H(\frac{\pi}{2})= H(\frac{3\pi}{2})=0$, so not sure whether the two direction rail line into the origin. Repeat the above steps to make $v=v$, $z_{2}=z_{3}v$, $dt_{3}=vdt_{2}$ into system\eqref{2.3.1.9}
\begin{equation}\label{2.3.1.10}
  \left\{
  \begin{array}{l}
    \frac{dv}{dt_{3}}=v^2-v^6z^2_{3}+\epsilon v^7z^2_{3}+v^9z^4_{3}-\mu v^{10}z^4_{3}-\beta v^{11}z^4_{3}, \\
    \frac{dz_{3}}{dt_{3}}=-2vz_{3}+2v^5z^3_{3}-2\epsilon v^6z^3_{3}-v^8z^5_{3}+2\mu v^9z^5_{3}+2\beta v^{11}z^5_{3}.
  \end{array}
  \right.
\end{equation}
By polar coordinate transformation $v=r\cos\theta$,~$z_{3}=r\sin\theta$, get $G(\theta)=-3\cos^2\theta \sin\theta$,~ $H(\theta)=\cos^3\theta-2\cos\theta\sin^2\theta$, such that the solution of the $G(\theta)=0$ is $\theta_{0}=0,~\pi,~\frac{\pi}{2},~\frac{3\pi}{2}$. Because of $G^{'}(0)H(0)=G^{'}(\pi)H(\pi)=-3<0$, so when $t_{3}\rightarrow\infty$ along $\theta=0$, $\theta=\pi$ has a unique track into the origin. Because $H(\frac {\pi}{2})=H(\frac{3\pi}{2})=0$, so not sure whether the two direction rail line into the origin. Although the result of this step is the same as the previous step, you can see that the minimum number of times has been reduced once again. To make $v=v$, $z_{3}=z_{4}v$, $dt_{4}=vdt_{3}$ into system\eqref{2.3.1.10}
\begin{equation}\label{2.3.1.11}
  \left\{
  \begin{array}{l}
    \frac{dv}{dt_{4}}=v-v^7z^2_{4}+\epsilon v^8z^2_{4}+v^9z^4_{4}-\mu v^{13}z^4_{4}-\beta v^{14}z^4_{4}, \\
    \frac{dz_{4}}{dt_{4}}=-3z_{4}+v^4z^2_{4}+2v^6z^3_{4}-v^8z^5_{4}-v^{11}z^5_{4}+3\beta v^{15}z^5_{4}+3\mu v^{13}z^5_{4}-3\epsilon v^7z^3_{4}.
  \end{array}
  \right.
\end{equation}
In this case $O(0,0)$ is a saddle of the system\eqref{2.3.1.11}. According to the characteristics of the Briot-Bouquet transformation, it can be concluded that $C(0,0)$ is an unstable node. And since the time transformation $dt=z^{4}d\tau$ has an even degree, So corresponding to Poincar\'{e} disc two equilibria $\bar{C}(0,1)$ and $\bar{\bar{C}}(0,1)$ stability and rail line to the same as $C(0,0)$.
\end{proof}

Next, we will discuss when the autonomous system\eqref{1.2.3} generates limit cycle. Below we discuss the existence of limit cycles for autonomous systems\eqref{1.2.3} in different parameter ranges.
For simplicity, the whole parameter space is divided into the following five subsets:
\begin{equation*}
  (c1):\left\{
  \begin{array}{lll}
    \epsilon  \leq0,\\
    \beta  >0,\\
    \mu  \in\mathbb{R},
  \end{array}
  \right.
  (c2):\left\{
  \begin{array}{lll}
    \epsilon  <0,\\
    \beta  =0,\\
    \mu  \in\mathbb{R},
  \end{array}
  \right.
  (c3):\left\{
  \begin{array}{lll}
    \epsilon  >0,\\
    \beta  =0,\\
    \mu  \leq-\frac{5}{36},
  \end{array}
  \right.
\end{equation*}
\begin{equation*}
  (c4):\left\{
  \begin{array}{lll}
    \epsilon  \in\mathbb{R},\\
    \beta  \in\mathbb{R},\\
    \mu  \leq-\frac{1}{4},
  \end{array}
  \right.
  (c5):\left\{
  \begin{array}{lll}
    \epsilon  >0,\\
    \beta  =0,\\
    \mu  \geq0,
  \end{array}
  \right.
  (c6):\left\{
  \begin{array}{lll}
    \epsilon  >0,\\
    \beta  >0,\\
    \mu  =\mu_{3},
  \end{array}
  \right.
  (c7):\left\{
  \begin{array}{lll}
    \epsilon  >0,\\
    \beta  >0,\\
    \mu  >\mu_{3}.
  \end{array}
  \right.
\end{equation*}
\begin{lemma}\label{L6}
When one of conditions (c1),(c2),(c3) and (c4) holds, System \eqref{1.2.3} exhibits neither limit cycles nor homoclinic loops.
\end{lemma}

\begin{proof}
When $\mu \leq - \frac{1}{4} $, we can calculate the system \eqref {1.2.3} divergence
\begin{align*}\label{2.3.3}
  \mathrm{div}\{y, (\mu+x^2-x^4)y+\beta x-\epsilon x^3\}=&\mu+x^2-x^4 \notag\\
  =&-(x^2-\frac{1}{2})^2+(\mu+\frac{1}{4})\notag\\
  \leq &~0,
\end{align*}
Therefore, according to the Bendixon-Dulac criterion, when $\mu\leq-\frac{1}{4}$, the system \eqref{1.2.3} has no closed orbit.

When $\epsilon\leq0$, $\beta>0$ and $\epsilon<0$, $\beta=0$, the finite equilibrium has only the origin $O(0,0)$, which is a saddle, if there is a limit cycle must contain this saddle, and according to the index theory, The system\eqref{1.2.3} has no limit cycle.

When $\epsilon>0 $, $\beta=0$, $\mu\leq-\frac{5}{36}$, equilibrium $O(0, 0)$ is a stable nodes
\begin{align*}
  E(x,y)=\int_{0}^{x}g(s)ds+\frac{y^2}{2}.
\end{align*}
Therefore
\begin{align*}
  \frac{dE(x,y)}{dt}\mid_{(1.7)}=-g(x)F(x)=\epsilon x^4(\mu+\frac{1}{3}x^2-\frac{1}{5}x^4)\leq0.
\end{align*}
If there is a limit cycle $\gamma$, then there is
\begin{align*}
  \oint_{\gamma}\frac{dE(x,y)}{dt}\mid_{(1.7)}=\oint_{\gamma}\epsilon x^4(\mu+\frac{1}{3}x^2-\frac{1}{5}x^4)dt<0.
\end{align*}
With the $\oint_{\gamma}\frac{dE(x,y)}{dt}=0$ contradiction, so the system \eqref{1.2.3} no limit cycle.

\end{proof}

\begin{lemma}
When the condition (c5) holds, there is a unique stable limit cycle for system\eqref{1.2.3}.
\end{lemma}
\begin{proof}
When $\epsilon>0, \beta=0, \mu\geq0$, by Lemma \ref{L4} we can know that $\bar{B}(1,0)$ and $\bar{\bar{B}}(1,0)$ are saddle, $\bar{C}(0,1)$ and $\bar{\bar{C}}(0,-1)$ are unstable nodes, so we choose the equator and these two pairs of diametral points together to form the outer boundary of the ring domain, and the orbitals point to the inner boundary. At this time, only one equilibrium at the origin on the finite plane is an unstable equilibrium (unstable node or unstable focus), and this equilibrium can be used as an inner boundary line, thus from the Poincar\'{e}-Bendixson ring domain theorem we can know that the system\eqref{1.2.3} has a stable limit cycle.
\end{proof}

Thus we prove that the system \eqref{1.2.3} has a large limit cycle with three equilibria when the parameter is in the range ~\normalsize{\textcircled{\scriptsize{1}}}~in the Theorem \ref{T2.2.0}.
Through the above analysis, we can give a representative global phase diagram of the unforced system\eqref{1.2.3}.

When $\epsilon\leq0$, the system\eqref{1.2.3} has only one equilibrium as the saddle, then there is no limit cycle, as shown in Figure \ref{F6}(a).

When $\epsilon>0, \beta=0, \mu\geq0$, system\eqref{1.2.3} has only one unstable equilibrium (unstable node or unstable focus), in which case there is a stable limit cycle, as shown in Figure \ref{F6}(b).

When $\epsilon>0, \beta>0, \mu\leq\mu_{c}$, system\eqref{1.2.3} has three equilibrium, where the origin $O(0,0)$ is the saddle and $E_{1,2}$ are stable foci, and there is no limit cycle, as shown in Figure \ref{F6}(c).

When $\epsilon>0, \beta>0, \mu_{c}<\mu<\mu_{3}$, system\eqref{1.2.3} has three equilibrium, in which $O(0,0)$ is a saddle, $E_{1,2}$ are unstable foci, and there is a stable limit cycle near $E_{1,2}$, as shown in Figure \ref{F6}(d).

When $\epsilon>0, \beta>0, \mu=\mu_{3}$, system\eqref{1.2.3} has three equilibrium, where the origin $O(0,0)$ is the saddle, $E_{1,2}$ are unstable foci, and there is a homohoming orbit at the saddle point, as shown in Figure \ref{F6}(e).

When $\epsilon>0, \beta>0, \mu>\mu_{3}$, system\eqref{1.2.3} has three equilibrium, where the origin $O(0,0)$ is a saddle, $E_{1,2}$ are unstable foci, and there is a stable large limit cycle containing three equilibria, as shown in Figure \ref{F6}(f).
\begin{figure}[htbp]
    \centering
    \subfigure[$\epsilon\leq0$]
        {
            \begin{minipage}{4cm}
                \centering
                \includegraphics[width=4 cm,height=4cm]{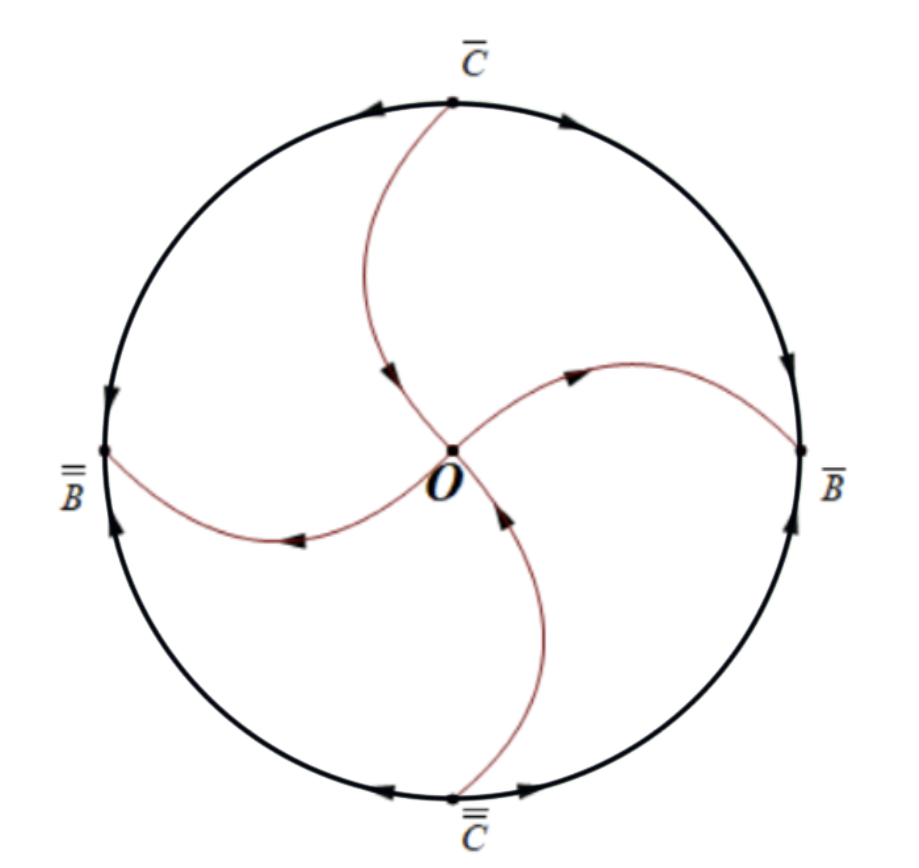}
            \end{minipage}
        }
    \subfigure[$\epsilon>0, \beta=0, \mu\geq0$]
        {
            \begin{minipage}{4cm}
                \centering
                \includegraphics[width=4 cm,height=4cm]{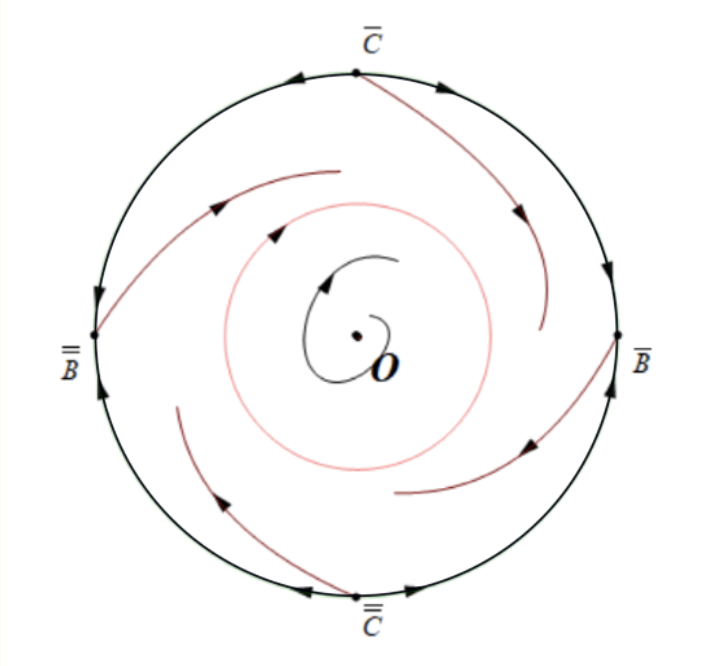}
            \end{minipage}
        }
    \subfigure[$\epsilon>0, \beta>0, \mu\leq\mu_{c}$]
        {
            \begin{minipage}{4cm}
                \centering
                \includegraphics[width=4 cm,height=4cm]{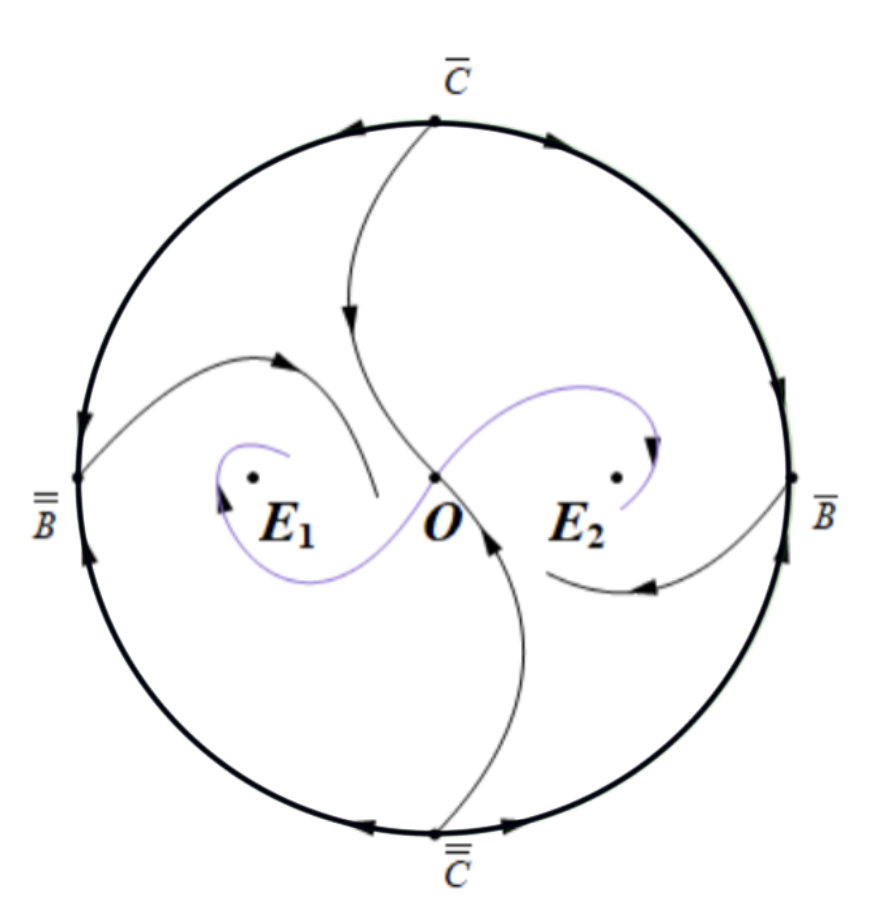}
            \end{minipage}
        }

    \subfigure[$\epsilon>0, \beta>0, \mu_{c}<\mu<\mu_{3}$]
        {
            \begin{minipage}{4cm}
                \centering
                \includegraphics[width=4 cm,height=4cm]{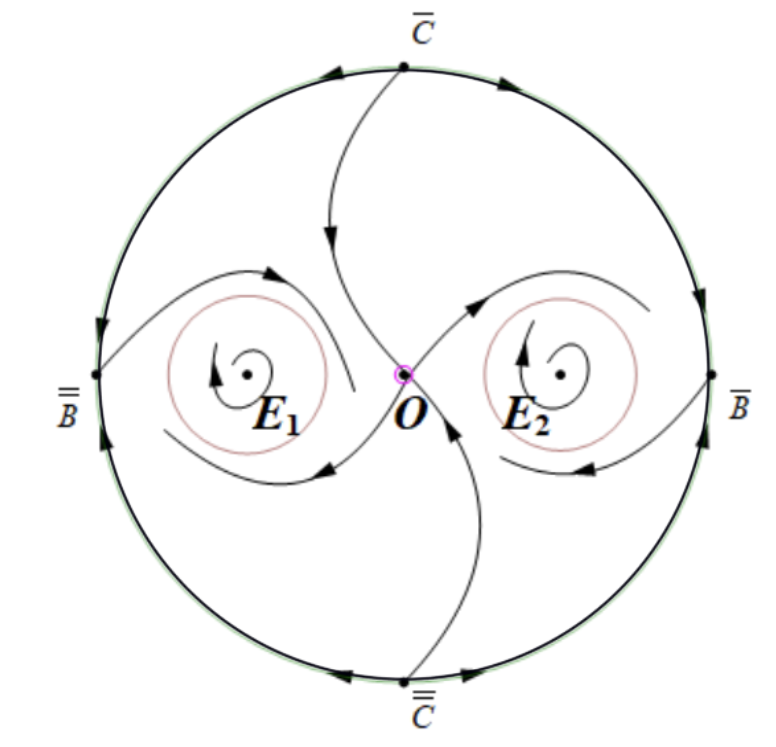}
            \end{minipage}
        }
    \subfigure[$\epsilon>0, \beta>0, \mu=\mu_{3}$]
        {
            \begin{minipage}{4cm}
                \centering
                \includegraphics[width=4 cm,height=4cm]{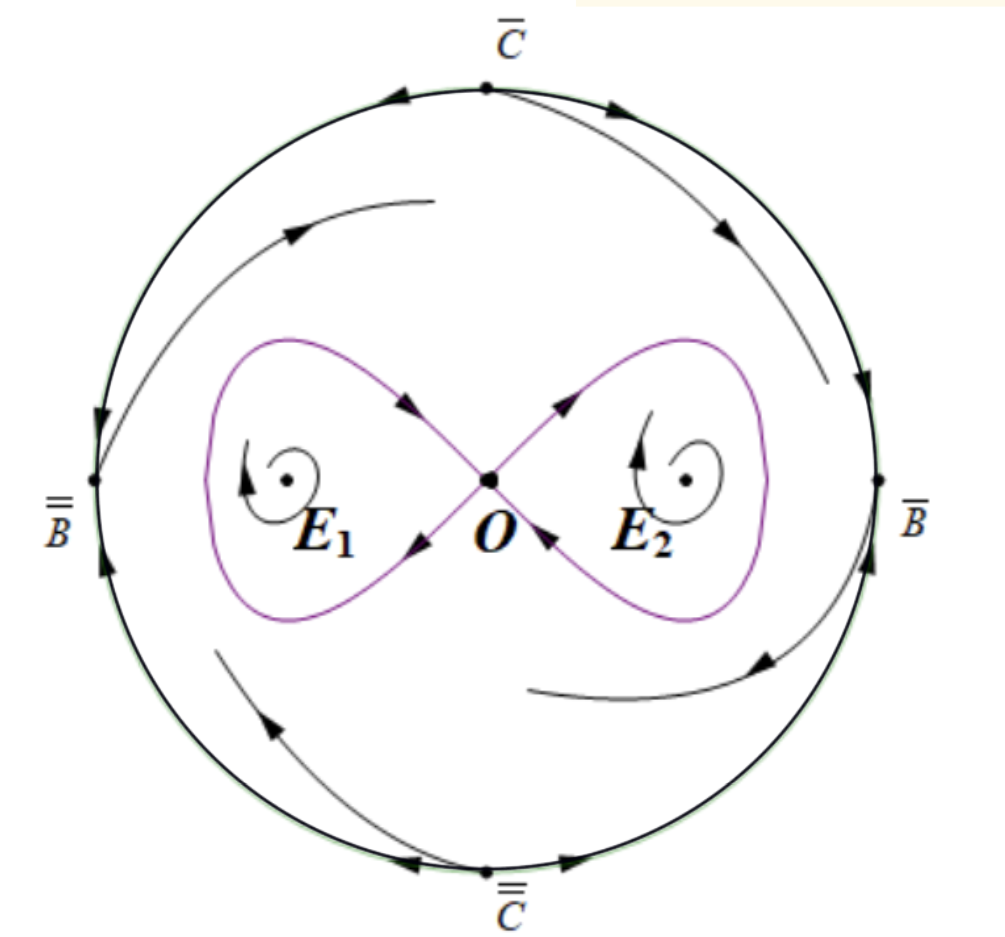}
            \end{minipage}
        }
     \subfigure[$\epsilon>0, \beta>0, \mu>\mu_{3}$]
        {
            \begin{minipage}{4cm}
                \centering
                \includegraphics[width=4 cm,height=4cm]{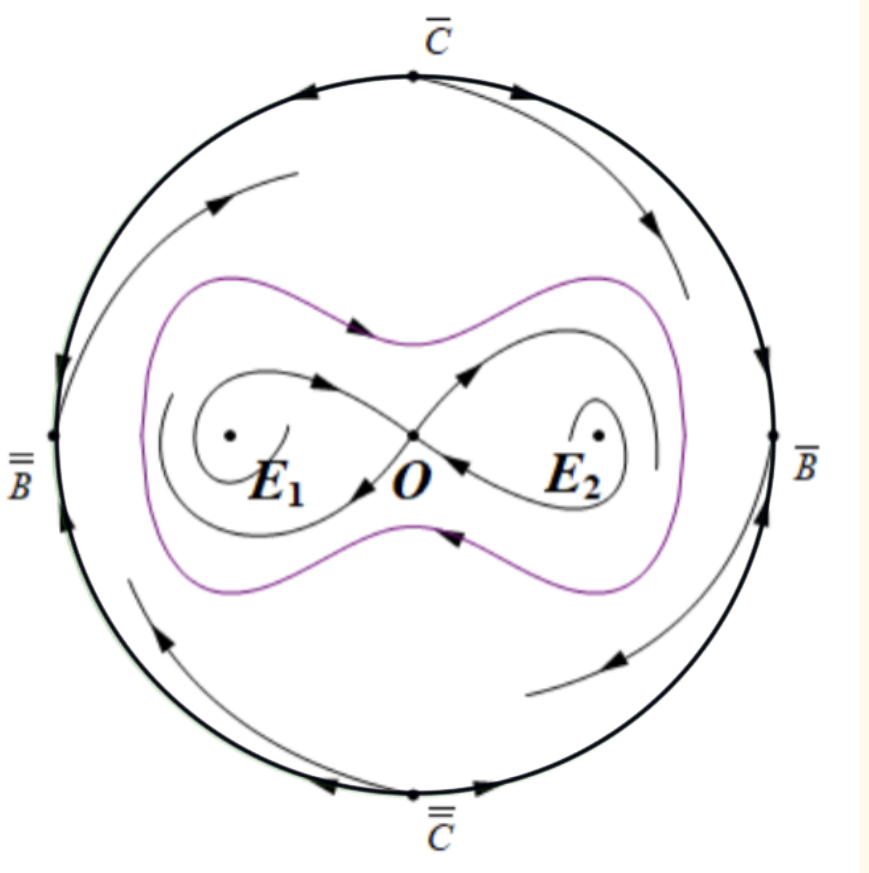}
            \end{minipage}
        }
    \caption{Global phase diagram for an unforced system \eqref{1.2.3}}\label{F6}
\end{figure}
\par
\section{The existence of periodic and quasi-periodic solutions for forced systems}
\label{chapter4}
\def\theequation{4.\arabic{equation}}
\setcounter{equation}{0}
\par
In the previous section, we studied the dynamics of autonomous systems. In this section, we will apply a KAM theory to discuss whether the original system has two-dimensional invariant torus. In the previous chapter, we know that the autonomous system\eqref{1.2.3} will generate Hopf bifurcation when the parameter $\mu=\mu_{c}$ is near the equilibria $E_{1,2}$, so we need to first shift the bifurcation parameter to the critical value, and then move the equilibrium to the origin. Then the system \eqref{1.2.2} is reduced to a standard form suitable for KAM theoretical analysis by a series of transformations similar to the Hopf bifurcation canonical form derived from autonomous systems. Finally, the existence of a two-dimensional quasi-periodic invariant torus of the system\eqref{1.2.2} is discussed by this KAM theorem. For convenience, we have relisted\eqref{1.2.2} here
\begin{equation}\label{3.1}
    \left\{
        \begin{array}{ll}
        \dot{x}=y,\\
        \dot{y}=(\mu+x^{2}-x^{4})y+\beta x-\epsilon x^{3}-\alpha x\cos(\omega t).
        \end{array}
    \right.
\end{equation}
\par
In the previous section, we know that the system\eqref{1.2.3} when the parameter $\mu=\mu_{c}$, the Hopf bifurcation will be generated near two equilibria: $E_{1},E_{2}$. Since $E_{1}$ is symmetric with the orbit near $E_{2}$, So just analyze the equilibrium $E_{2}(x_{2},0)$, where $x_{2}=\sqrt{\frac{\beta}{\epsilon}}$. We get the standard form of autonomous system by a series of transformations. Due to the addition of periodic external forces, we guess that the system\eqref{3.1} will produce quasi-periodic solutions of two fundamental frequencies near $E_{2}$. The standard form will be obtained by a series of transformations similar to the previous section. First, translate the bifurcation parameters so that $\mu=\mu_{c}+\xi$,~$\xi$ are small perturbed parameters, then the system\eqref{1.2.3} becomes
\begin{equation}\label{3.2.1}
  \left\{
  \begin{array}{ll}
    \dot{x}=y, \\
  \dot{y}=(\mu_{c}+\xi+x^2-x^4)y+\beta x-\epsilon x^3.
  \end{array}
  \right.
\end{equation}
Then we apply the translation transformation \eqref{2.2.3} to move $E_{2}(x_{2},0)$ to the origin,
\begin{equation}\label{3.2.2}
  \left\{
  \begin{array}{ll}
    \dot{x}=&y, \\
  \dot{y}=&-2\beta x+\xi y-3\epsilon x_{1}x^2+\frac{2x_{1}(-2\beta+\epsilon)}{\epsilon}xy\\
  &-\epsilon x^3+(1-6x^2_{1})x^2y-4x_{1}x^3y-x^4y.
  \end{array}
  \right.
\end{equation}
The corresponding Jacabian matrix of the origin is
\begin{equation*}
  \left(
  \begin{array}{cc}
    0 & 1 \\
    -2\beta & \xi \\
  \end{array}
\right),
\end{equation*}
Its corresponding eigenvalue is
\begin{align*}
  \lambda=&\frac{1}{2}(\xi+{\rm i}\sqrt{4\beta-\xi^2}), \\
  \bar{\lambda}=&\frac{1}{2}(\xi-{\rm i}\sqrt{4\beta-\xi^2}).
\end{align*}
After a series of transformations \eqref{2.2.5}, \eqref{2.2.7}, \eqref{2.2.8} similar to the transformation in the previous chapter \eqref{3.2.2} can be written
\begin{equation}\label{3.2.4}
  \dot{z}=\lambda(\xi) z+c_{1}(\xi)z^2\bar{z}+O(|z|^4).
\end{equation}
Next order
\begin{align*}
  z=\sqrt{\rho}e^{{\rm i}\varphi},
\end{align*}
then the system\eqref{3.2.4} can be reduced to
\begin{equation}\label{3.2.6}
  \left\{
  \begin{array}{ll}
    \dot{\rho}&=2\Re\lambda(\xi)\rho+2\Re c_{1}(\xi)\rho^2+P_{1}(\rho,\varphi,\xi), \\
  \dot{\varphi}&=\Im\lambda(\xi)+\Im c_{1}(\xi)\rho+P_{2}(\rho,\varphi,\xi),
  \end{array}
  \right.
\end{equation}
where
\begin{align*}
  \Re\lambda(\xi)=&\frac{1}{2}\xi,\\
  \Im\lambda(\xi)=&\frac{1}{2}(\sqrt{4\beta-\xi^2})=\sqrt{\beta}+\frac{1}{4\sqrt{\beta}}\xi+O(\xi^2),
\end{align*}
and $P_{1}(\rho,\varphi,\xi)$ is a smooth function about $\rho$ starting from the cubic term, and $P_{2}(\rho,\varphi,\xi)$ is a smooth function about $\rho$ starting from the quadratic term.

{\rem We provide the following concrete example to show that the following scaling transformation of the parameters and the system can be achieved. Fixed $\beta=1$, $\epsilon=1$ then $\mu_{c}=0$, can be calculated
\begin{align*}
  &x_{1} = -1, \\
  &\Re\lambda(\xi) =0.5\xi, \\
  &\Im\lambda(\xi) =1+0.25\xi+O(\xi^2),\\
  &c_{1}(\xi)\approx (2-1.65885\mathrm{i})+(2.94835-0.625\mathrm{i})\xi+O(\xi^2).
\end{align*}}

For the truncation equation $2\Re\lambda(\xi)\rho+2\Re c_{1}(\xi)\rho^2$ of the polar radius of the system\eqref{3.2.6}, it has a non-zero equilibrium solution
\begin{align*}
  \rho_{0}=-\frac{\Re\lambda(\xi)}{\Re c_{1}(\xi)}.
\end{align*}
Since $\xi$ is a small perturbation parameter, we can set $\xi\rightarrow\varepsilon\xi$. Obviously $\rho_{0}=\frac{1}{4}\varepsilon\xi+O(\varepsilon^2\xi^2)=O(\varepsilon)$, make a translation transformation
\begin{align*}
  \rho=\varepsilon^{\frac{3}{2}}I+\rho_{0},
\end{align*}
then the system\eqref{3.2.6} can be written as
\begin{equation}\label{3.2.7}
  \left\{
  \begin{array}{ll}
    \dot{I}=&-2\Re\lambda(\xi)I\varepsilon+2\Re c_{1}(\xi)I^2\varepsilon^{\frac{3}{2}}+\varepsilon^{\frac{5}{2}}\widetilde{P_{1}}(I,\varphi,\xi), \\
  \dot{\varphi}=&\Im\lambda(\xi)+\Im c_{1}(\xi)\rho_{0}+\varepsilon^{\frac{3}{2}}\Im c_{1}(\xi)I+\varepsilon^{\frac{3}{2}}\widetilde{P_{2}}(I,\varphi,\xi),
  \end{array}
  \right.
\end{equation}
denote as
\begin{equation}\label{3.2.8}
  \left\{
  \begin{array}{ll}
    \dot{I}&=\varepsilon(H_{1}(\xi)I+\varepsilon^{\frac{1}{2}}\widetilde{G_{1}}(I,\varphi,\xi)), \\
  \dot{\varphi}&=w_{1}(\xi)+\varepsilon^{\frac{3}{2}}\widetilde{G_{2}}(I,\varphi,\xi),
  \end{array}
  \right.
\end{equation}
where
\begin{align*}
  H_{1}(\xi)&=-\xi, \\
  w_{1}(\xi)&=\Im\lambda(\varepsilon\xi)-\frac{\Re\lambda(\varepsilon\xi)\Im(\varepsilon\xi)}{\Re c_{1}(\varepsilon\xi)}\\
  &=w_{10}+w_{11}(\xi)\varepsilon+O(\varepsilon^2),
\end{align*}
and $\widetilde{G_{1}}(I,\varphi,\xi)$, $\widetilde{G_{2}}(I,\varphi,\xi)$ all about $I,\varphi$ is real analytic, and about $\xi$ is sufficiently smooth on some bounded closed set.

{\rem In order to facilitate the proof of the following theorem, specific examples are given: fixed $\beta=1$, $\epsilon=1$, then $\mu_{c}=0$, we can calculate it
\begin{align*}
  &\rho_{0}\approx -0.5\varepsilon\xi, \\
  &H_{1}(\xi)=-\xi, \\
  &w_{1}(\xi,\varepsilon)\approx 1+0.82943\varepsilon\xi+O((\varepsilon\xi)^2).
\end{align*}}

Now consider the original system\eqref{3.1}, because $\alpha$is a small quantity, can make $\alpha\rightarrow\varepsilon^3\alpha$, $\psi=\omega t$, after the above similar series of transformations, The system\eqref{3.1} can be written as
\begin{equation}\label{3.2.9}
  \left\{
  \begin{array}{ll}
    \dot{I}&=\varepsilon[H_{1}(\xi)I+\varepsilon^{\frac{1}{2}}G_{1}(I,\varphi,\psi,\xi,\alpha,\varepsilon)],\\
    \dot{\varphi}&=w_{1}(\xi,\varepsilon)+\varepsilon^{\frac{3}{2}}G_{2}(I,\varphi,\psi,\xi,\alpha,\varepsilon),\\
    \dot{\psi}&=\omega.
  \end{array}
  \right.
\end{equation}
and $G_{1}(I,\varphi,\psi,\xi)$, $G_{2}(I,\varphi,\psi,\xi)$ are all about $I,~\varphi,~\psi$ real analytic, and about $\xi$ is sufficiently smooth in the region of $\xi>0$. For reduced equation\eqref{3.2.9} we have the following theorem.

{\theorem \label{D1}Hypothesis $\xi\in\Pi=[\frac{1}{16},1]$, So for the given $0<\gamma_{0}\ll1$, there are sufficiently small positive numbers $\varepsilon_{0}^{*}$,
such that for $0<\varepsilon<\varepsilon_{0}^{*}$, $\varepsilon_{0}^{*}=o(\gamma^{4}_{0})$, there is a Cantor subset $\Pi_{\gamma_{0}}\subset[\frac{1}{16},1]$,
for any of $\xi\in\Pi_{\gamma_{0}}$, the system\eqref{3.2.9} exist two quasi-periodic solutions with the fundamental frequency $(\omega_{1}^{*},\omega)$, $|\omega_{1}^{*}-\omega_{1}|=O(\varepsilon^{\frac{3}{2}})$
and when $\gamma_{0}\rightarrow0$, we have ${\rm meas}(\Pi\setminus\Pi_{\gamma_{0}})\rightarrow0$.}
\par
\begin{proof}
In section 2 we know that the autonomous system\eqref{1.2.3} has a small limit cycle in the sufficiently small right neighborhood of the bifurcation point $\mu=\mu_{c}$ (i.e. $\xi=0$), and because we have scaled $\xi\rightarrow\varepsilon\xi$ for the parameter $\xi$, so long as $\varepsilon$ is small enough to value $\xi$ in a closed set, $\varepsilon\xi$ is in the sufficiently small right neighborhood of $\xi=0$. Since the KAM theorem is used to analyze the existence of quasi-periodic solutions, it is usually necessary to restrict $\xi$ to a closed interval leaving $\xi=0$, we may as well take $\xi\in[\frac{1}{16},1]$ and let $\Pi=[\frac{1}{16},1]$.
System\eqref{3.2.9} is a special case that takes $n_{11}=n_{21}=0$,~~$n_{12}=1$,~~$n_{22}=2$,~~$q_{1}=q_{7}=\frac{3}{2}$,~~$q_{2}=q_{5}=\frac{1}{2}$,~~$q_{3}=q_{4}=q_{6}=1$ , make
\begin{align*}
  I_{2}=I,~~\varphi_{2}=(\varphi,\psi)^{T},
\end{align*}
generation into the reduced equation\eqref{3.2.9} can be written
\begin{equation}\label{3.3.1}
  \left\{
  \begin{array}{ll}
    \dot{I}_{2}&=\varepsilon[A_{2}(\xi)I_{2}+\varepsilon^{\frac{1}{2}}g_{2}(I_{2},\varphi_{2},\xi,\varepsilon)],\\
    \dot{\varphi}_{2}&=w(\xi,\varepsilon)+\varepsilon^{\frac{3}{2}}g_{4}(I_{2},\varphi_{2},\xi,\varepsilon),\\
  \end{array}
  \right.
\end{equation}
where
\begin{align*}
  A_{2}(\xi)=-\xi,
\end{align*}
\begin{equation*}
  w^{0}(\xi,\varepsilon)=
  \left(
  \begin{array}{cc}
     w_{1}(\xi,\varepsilon) &  \\
     \omega
   \end{array}
   \right).
\end{equation*}

The assumptions conditions (H1)-(H3) and the nondegenerativity condition with respect to frequency are verified below. Since the system\eqref{3.3.1} is a special case when [\cite{L35},Theorem 2.2] takes $n_{11}=n_{21}=0$, $n_{12}=1$, $n_{22}=2$, $q_{1}=q_{7}=\frac{3}{2}$, $q_{2}=q_{5}=\frac{1}{2}$, $q_{3}=q_{4}=q_{6}=1$, it clearly satisfies (H1).

Since $g_{2}$,~$g_{4}$ is continuously differentiable with respect to the arguments and smooth with respect to the coordinate variables of any order, and $n_{2}=n_{21}+n_{22}=2$, so it is desirable to $l=30$,~$\alpha=1$,~$\iota=3$. Obviously $g_{i}\in C^{l,\alpha}(\Omega\times\mathbb{T}^{n_{2}},~\Pi_{0})(i=2,4)$, $l>2(\alpha+1)(\iota+2)+\alpha\iota,~\iota>\alpha n_{2}-1$, that is, the hypothesis (H3) is satisfied.

For $\xi\in\Pi=[\frac{1}{16},1]$, $\inf\limits_{\xi\in\Pi}|A_{2}(\xi)|=\frac{1}{16}$, therefore, according to the existence theorem of implicit functions, there exist positive constants $c_{0},~c_{1}$ and $\varepsilon^{*}$ such that for any $\varepsilon\in(0,\varepsilon^{*}]$ satisfying
\begin{align*}
  \inf\mid\lambda(\xi)\mid=&\inf\mid A_{2}(\xi)\mid  \geq0,\\
  \parallel B_{2}\parallel_{1;\Pi},~ \parallel B^{-1}_{2}\parallel_{1;\Pi},&
  \parallel w^{0}\parallel_{0;\Pi}:=\sup\limits_{\xi\in\Pi}\mid w^{0}\mid  \leq c_{1},
\end{align*}
It is known from Remark 2
\begin{align*}
  \parallel \partial_{\xi} w^{0}\parallel_{\Pi}:=\sup\limits_{\xi\in\Pi}\mid \partial_{\xi} w^{0}(\xi,\varepsilon)\mid\leq c_{1}\varepsilon.
\end{align*}
So hypothesis (H2) is true.

That is, assumption conditions (H1)-(H3) in [\cite{L35},Theorem 2.2], standard type\eqref{3.3.1} are satisfied, so for a given $0<\gamma_{0}\ll 1$, there exists a sufficiently small positive number $\varepsilon^{*}$ such that for $0<\varepsilon\leq\varepsilon^{*}$, $\varepsilon^{*}=o(\gamma^{4}_{0})$, there is a subset of Cantor $\Pi_{\gamma_{0}}\in\Pi$, For any $\xi\in\Pi_{\gamma_{0}}$, the system\eqref{3.3.1} has a quasi-periodic solution and an estimate is given. Since the frequency mapping does not satisfy the conditions in the measure estimation of [\cite{L35},Theorem 2.3], so [\cite{L35},Theorem 2.3] cannot be directly applied. It is shown below that the measure estimate of $\Pi_{\gamma 0}$, $\Pi_{\gamma 0}$ is the set of $\Pi$ by removing some parameters that make the denominator too small, because
\begin{align*}
  \omega^{v}&=w^{0}+O(\varepsilon^{\frac{3}{2}}),\Pi_{v}=\Pi_{v-1}\backslash \bigcup_{k}\mathfrak{R}^{v}_{k}(\gamma_{0}),\\
  k\in&\mathbb{Z}^{2}\setminus\{0\}, K_{v-1}<|k|_{2}\leq K_{v},~v=1,2,\cdots
\end{align*}

such that $\Pi_{\gamma 0}=\bigcap\limits_{v=0}^{\infty}\Pi_{v}$, where $\Pi_{0}=\Pi$,
\begin{align}\label{3.3.2}
  \mathfrak{R}_{k}^{v}(\gamma_{0})=\{\xi:|(k,\omega^{v})|<\frac{\varepsilon \gamma_{0}}{|k|^3_{1}}\}.
\end{align}
$\gamma_{0}=\varepsilon^{\kappa},~0<\kappa\leq\frac{1}{4}$. Dig out the parameter set $\bigcup_{k}\mathfrak{R}^{v}_{k}(r_{0})$, measure estimation of $\mathfrak{R}^{v}_{k}(r_{0})$ is performed below, $k=(k_{1},~k_{2})$, $|k|_{1}=|k_{1}|+|k_{2}|$, Substituting the expression $\omega^{v}$ into \eqref{3.3.2} yields
\begin{align*}
 |(k,\omega^{v})|=|k_{1}w_{1}(\xi,\varepsilon)+k_{2}\omega|=|k_{1}(w_{10}+w_{11}(\xi)\varepsilon+O(\varepsilon^2))+k_{2}\omega|<\frac{\varepsilon \gamma_{0}}{|k|^3_{1}},
\end{align*}
where $w_{10}=w_{1}(\xi,\varepsilon)$, $w_{11}=\partial_{\xi}w_{1}|_{\varepsilon=0}$, and $w_{10}$ is a non-zero constant, $\inf\limits_{\xi\in\Pi}|w_{11}|>0$.
make
\begin{equation*}\label{3.3.3}
  \phi(\xi)=|k_{1}(w_{10}+w_{11}(\xi)\varepsilon+O(\varepsilon^2))+k_{2}\omega|.
\end{equation*}
The following analysis needs to be discussed in several different cases:

(i)If $k_{1}w_{10}+k_{2}\omega=0$, Since $k\neq0$, $k_{1}\neq0$, and $w_{10},~\omega$ are constants, so for sufficiently small $\gamma_{0}$, there is
\begin{align*}
 |\phi(\xi)|=|k_{1}(w_{11}(\xi)\varepsilon+O(\varepsilon^2))|\geq\frac{\varepsilon \gamma_{0}}{|k|^3_{1}},
\end{align*}
so $\mathfrak{R}^{v}_{k}(\gamma_{0})={\emptyset}$.

(ii)If $k_{1}w_{10}+k_{2}\omega\neq0$;

(a)If $k_{1}=0$, then $k_{2}\neq0$, then there is
\begin{align*}
  |\phi(\xi)|=|k_{2}\omega|\geq\frac{\varepsilon \gamma_{0}}{|k|^3_{1}},
\end{align*}
so $\mathfrak{R}^{v}_{k}(\gamma_{0})={\emptyset}$.

(b)If $k_{1}\neq0$, since $w_{10}$ and $\omega$ have nothing to do with $\xi$, the derivative of $\phi(\xi)$ with respect to $\xi$ can be obtained as
\begin{align*}
  \frac{d\phi(\xi)}{d\xi}=k_{1}\varepsilon\frac{dw_{11}(\xi)}{d\xi}+O(\varepsilon^2),
\end{align*}
so, by
\begin{align*}
  \mathrm{meas}\{\xi\in\Pi:~|\phi(\xi)|<a\}\leq\frac{2a}{\inf\limits_{\xi\in\Pi}|\phi^{'}(\xi)|},
\end{align*}
and $k_{1}\geq1$, Remark 3 It follows that $\frac{dw_{11}(\xi)}{d\xi}$ has a positive lower bound, so there exists $c_{2}$ such that
\begin{align*}
  \mathrm{meas}~\mathfrak{R}^{v}_{k}(\gamma_{0})\leq c_{2}\frac{r_{0}}{|k|^3_{1}}.
\end{align*}
Combined with the above analysis, we can get
\begin{align*}
  \mathrm{meas}~(\bigcup_{v, k}\mathfrak{R}^{v}_{k}(\gamma_{0}))\leq c_{2}r_{0}\sum_{0\neq k\in\mathbb{Z}^2}\frac{1}{|k|^3_{1}}.
\end{align*}
And because $\sum\limits_{0\neq k\in\mathbb{Z}^2}\frac{1}{|k|^3_{1}}$ converges, there exists $c_{3}$, such that
\begin{align*}
   \mathrm{meas}~(\bigcup_{v,k}\mathfrak{R}^{v}_{k}(\gamma_{0}))\leq c_{3}\gamma_{0}.
\end{align*}
Then there is
\begin{align*}
  \Pi_{\gamma0}=\Pi\setminus(\bigcup_{v,k}\mathfrak{R}^{v}_{k}(\gamma_{0})),
\end{align*}
so
\begin{align*}
  \mathrm{meas}~\Pi_{\gamma_{0}}=\mathrm{meas}~\Pi-O(\gamma_{0}).
\end{align*}

Then for sufficiently small $\gamma_{0}$, the Cantor set $\Pi_{\gamma_{0}}$ defined in theorem\ref{D1} has a positive Lebesgue measure, and when $\gamma_{0}\rightarrow 0$, there is $\mathrm{meas}~(\Pi\setminus\Pi_{\gamma0})\rightarrow0$.\\
\end{proof}

Since the previous transformations are invertible, it can be seen from the Theorem \ref{D1} that the system\eqref{1.2.2} has quasi-periodic solutions. Obviously, a periodic solution is generated near the saddle, $O(0,0)$.

\section{Numerical simulations}
\par

In this section we will give the phase diagram of system\eqref{1.2.3} by numerical simulation and show the bifurcation. For simplicity, we use $UM$ and $SM$ to represent unstable and stable manifolds in the simulation phase diagram. The qualitative properties of the system\eqref{1.2.3} at infinity cannot be reflected in numerical simulations.
Since when $\epsilon\leq0$ the system has only one equilibrium that origin is a saddle and the phase diagram structure is relatively simple, we consider the case where $\epsilon>0$ in the following.

Example 1: When $\epsilon=2$, $\beta=0$, $\mu=0$, the system has only one equilibrium, which the origin is a unstable focus, and there is a stable limit cycle. See the Figure 2(a).

Example 2: When $\epsilon=2$, $\beta=0$, $\mu=1$, the system has only one unstable equilibrium that origin, and there is a stable limit cycle. See the Figure 2(b).

Example 3: When $\epsilon=2$, $\beta=1$, $\mu=-0.25$, the system has three equilibria, the origin is saddle, $E_ {1, 2} $ are stable foci. See the Figure 2(c).

Example 4: When $\epsilon=2$, $\beta=1$, $\mu=0.2$, the system has three equilibria, the origin is saddle, $E_{1,2}$ are unstable foci, and in the $E_{1,2}$ near have a stable limit cycle. See the Figure 2(d).

Example 5: When $\epsilon=2$, $\beta=1$, $\mu=-0.171$, the system has three equilibria, where the origin is saddle, $E_{1,2}$ are the unstable foci, and there is a homocyclic ring that resides at the saddle. See the Figure 2(e).

Example 6: When $\epsilon=2$, $\beta=1$, $\mu=-0.1$, the system has three equilibria, where the origin is a saddle, $E_{1,2}$ are unstable foci, and there is a large limit cycle containing three equilibria. See the Figure 2(f).
\begin{figure}[htbp]
    \centering
    \subfigure[$\epsilon=2$, $\beta=0$, $\mu=0$]
        {
            \begin{minipage}{6cm}
                \centering
                \includegraphics[width=6 cm,height=6cm]{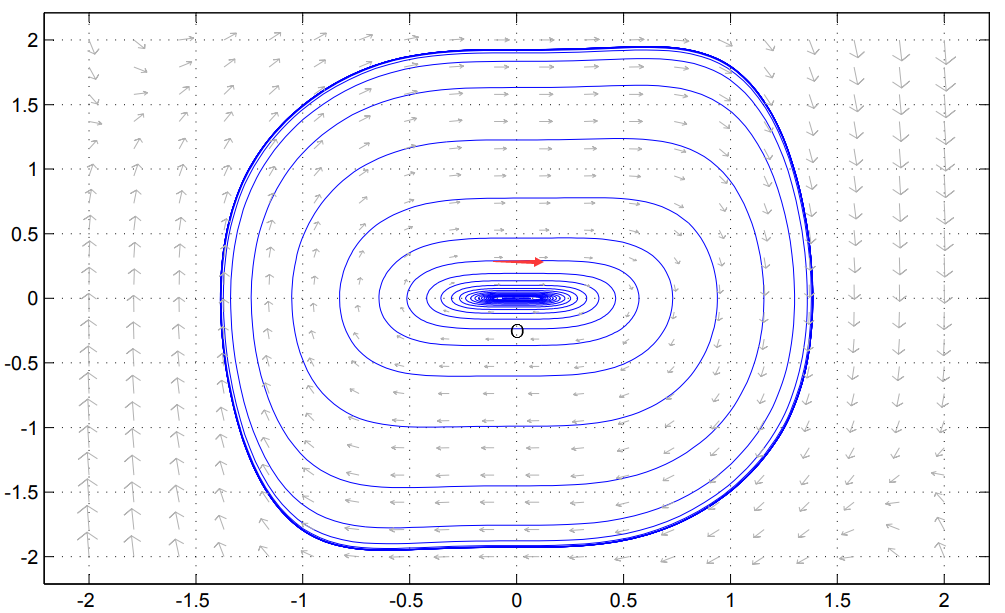}
            \end{minipage}
        }
    \subfigure[$\epsilon=2$, $\beta=0$, $\mu=1$]
        {
            \begin{minipage}{6cm}
                \centering
                \includegraphics[width=6 cm,height=6cm]{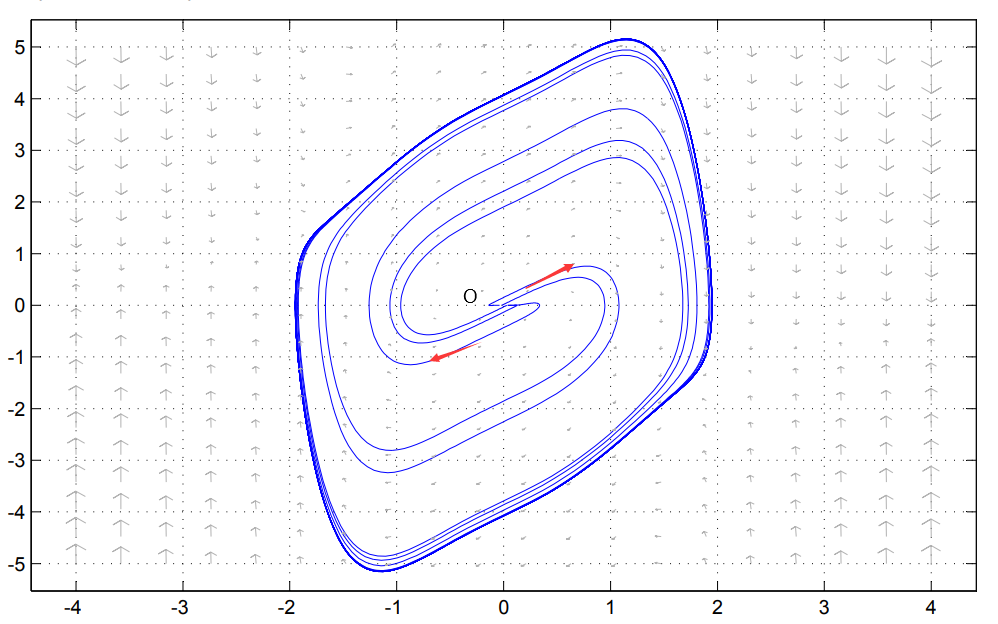}
            \end{minipage}
        }

    \subfigure[$\epsilon=2$, $\beta=1$, $\mu=-0.25$]
        {
            \begin{minipage}{6cm}
                \centering
                \includegraphics[width=6 cm,height=6cm]{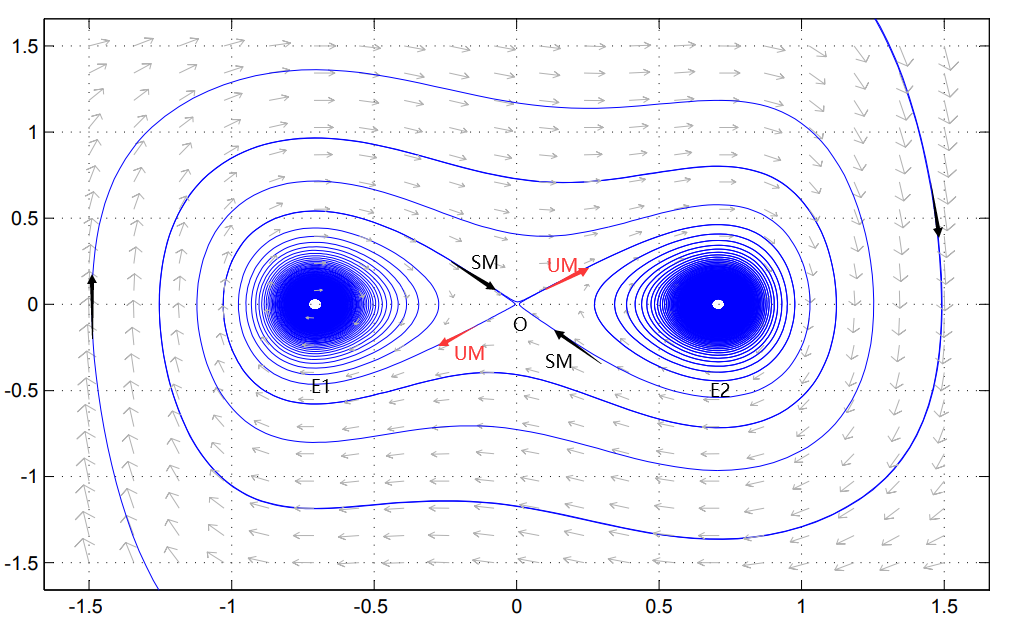}
            \end{minipage}
        }
    \subfigure[$\epsilon=2$, $\beta=1$, $\mu=-0.2$]
        {
            \begin{minipage}{6cm}
                \centering
                \includegraphics[width=6 cm,height=6cm]{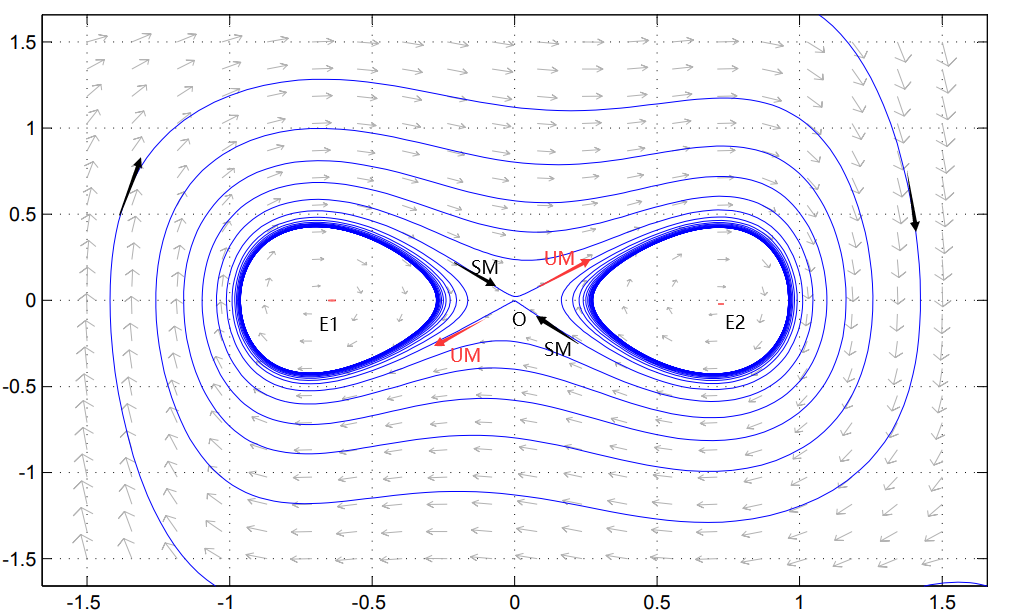}
            \end{minipage}
        }

    \subfigure[$\epsilon=2$, $\beta=1$, $\mu=-0.171$]
        {
            \begin{minipage}{6cm}
                \centering
                \includegraphics[width=6 cm,height=6cm]{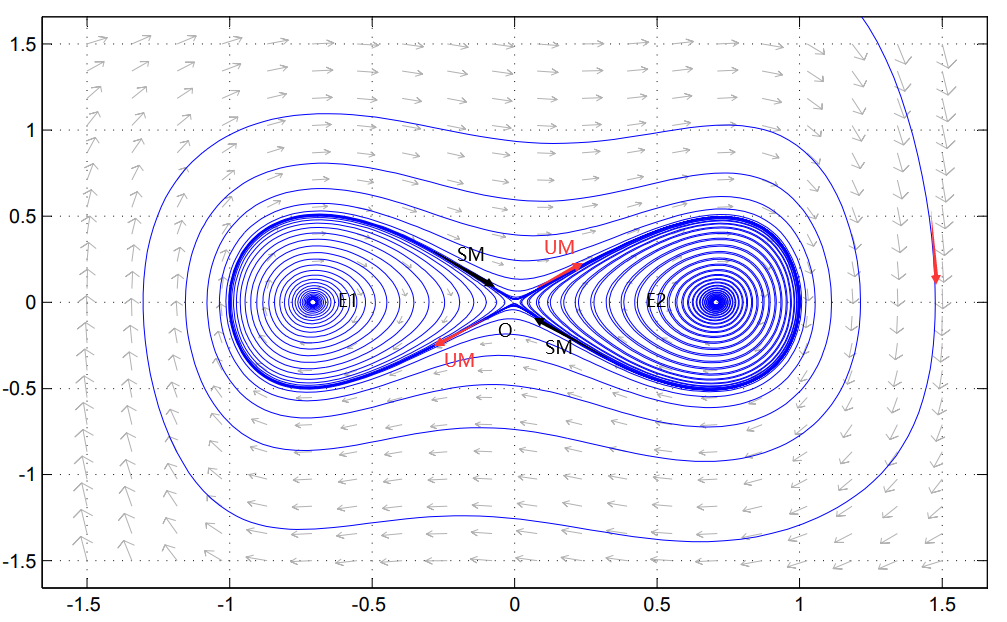}
            \end{minipage}
        }
    \subfigure[$\epsilon=2$, $\beta=1$, $\mu=-0.1$]
        {
            \begin{minipage}{6cm}
                \centering
                \includegraphics[width=6 cm,height=6cm]{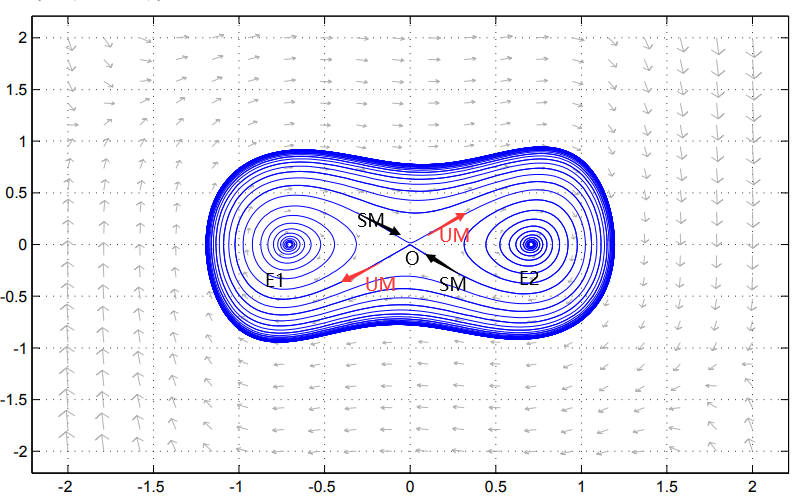}
            \end{minipage}
        }
    \caption{Numerical simulation of the unforced system \eqref{1.2.3}}
\end{figure}
\par

Next, we will show the phase diagram of the system\eqref{3.1} by numerical simulation.
Take $\beta=1$, $\epsilon=3$, $\omega=1$, $\alpha=-0.3$. When $\mu\leq-\frac{2}{9}$, there is no quasi-periodic solution for the system \eqref{3.1}, as shown in Figure \ref{F3}(a) is $\mu=-0.3$. When $\mu>-\frac{2}{9}$, the autonomous system\eqref{1.2.3} has a limit cycle caused by Hopf bifurcation, and when the amplitude $\alpha$ changes in a small range, the system\eqref{3.1} has a two-dimensional torus. The phase diagram of the system\eqref{3.1} when Figure \ref{F3}(b) is $\mu=-0.1$.
\begin{figure}[htbp]
    \centering
    \subfigure[Phase diagram of the system \eqref{3.1} in the $(x,y)$ plane when $\mu= ? 0.3$, $\beta=1$, $\epsilon=3$, $\omega=1$, and $\alpha= ? 0.3$]
        {
            \begin{minipage}{6cm}
                \centering
                \includegraphics[width=6 cm,height=6cm]{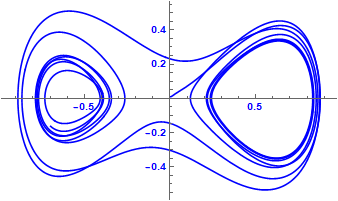}
            \end{minipage}
        }~~~~
    \subfigure[Phase diagram of the system \eqref{3.1} in the $(x,y)$ plane when $\mu=-0.1$, $\beta=1$, $\epsilon=3$, $\omega=1$, $\alpha=-0.3$]
        {
            \begin{minipage}{6cm}
                \centering
                \includegraphics[width=6 cm,height=6cm]{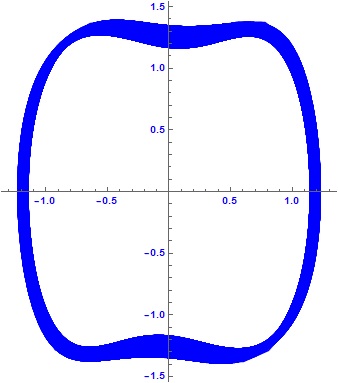}
            \end{minipage}
        }
    \caption{Numerical simulation of forced system \eqref{3.1}}\label{F3}
\end{figure}

Figure \ref{F5} is the oscillatory graph of $x$ and $y$ changing with time $t$ respectively when $\mu=-0.3$. We can see that the changes of $x$ and $y$ are irregular at this time, that is, there is no quasi-periodic solution at this time.
\begin{figure}[htbp]
    \centering
    \subfigure[$x$ changes track with time $t$]
        {
            \begin{minipage}{6cm}
                \centering
                \includegraphics[width=6 cm,height=6cm]{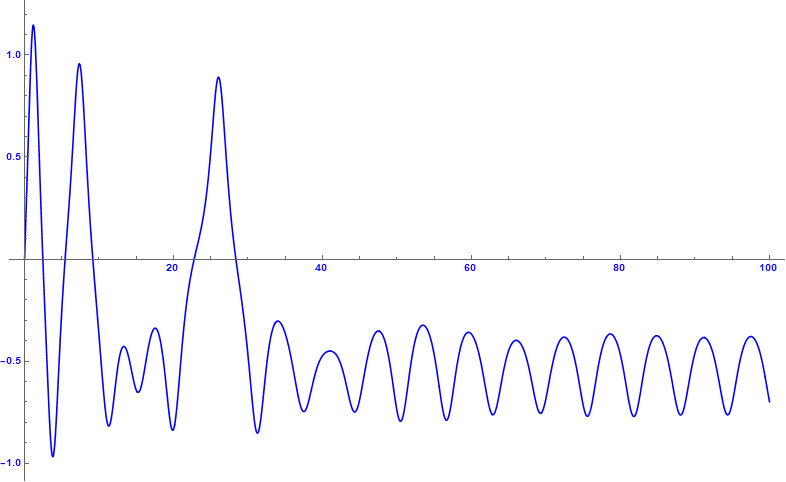}
            \end{minipage}
        }
    \subfigure[$y$ changes track with time $t$]
        {
            \begin{minipage}{6cm}
                \centering
                \includegraphics[width=6 cm,height=6cm]{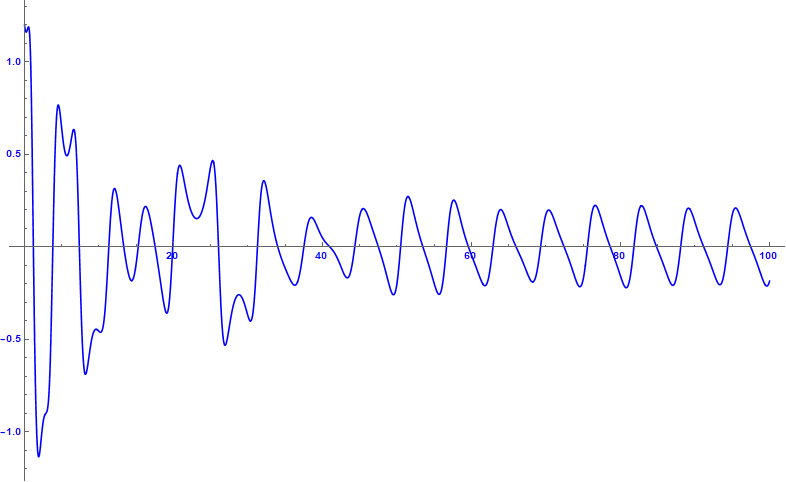}
            \end{minipage}
        }
    \caption{When $\mu=-0.3$, the system \eqref{3.1} takes the initial value $(x_{0},y_{0})=(0,1.2)$as the oscillation diagram of the change of $x$and $y$with time $t$, respectively}\label{F5}
\end{figure}

Figure \ref{F4} shows the oscillation diagram of $x$ and $y$ changing with time $t$ respectively when $\mu=-0.1$. We can see that the changes of $x$ and $y$ show a certain rule, that is, there is a quasi-periodic solution at this time.
\begin{figure}[htbp]
    \centering
    \subfigure[$x$ changes track with time $t$]
        {
            \begin{minipage}{6cm}
                \centering
                \includegraphics[width=6 cm,height=6cm]{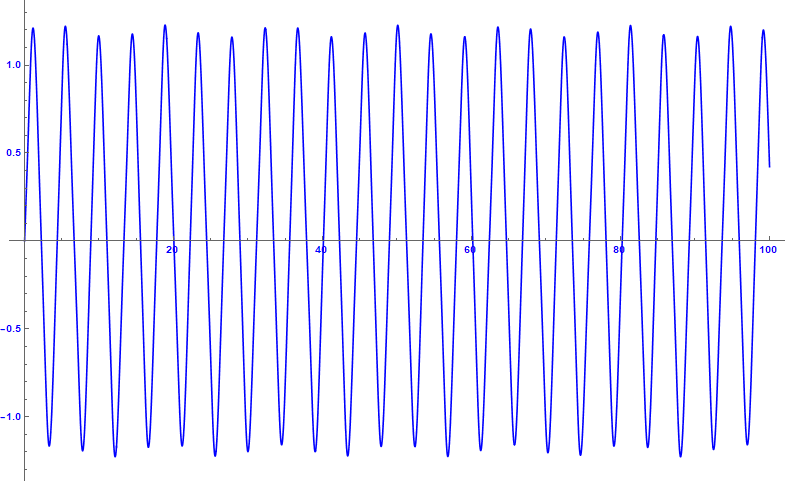}
            \end{minipage}
        }
    \subfigure[$y$ changes track with time $t$]
        {
            \begin{minipage}{6cm}
                \centering
                \includegraphics[width=6 cm,height=6cm]{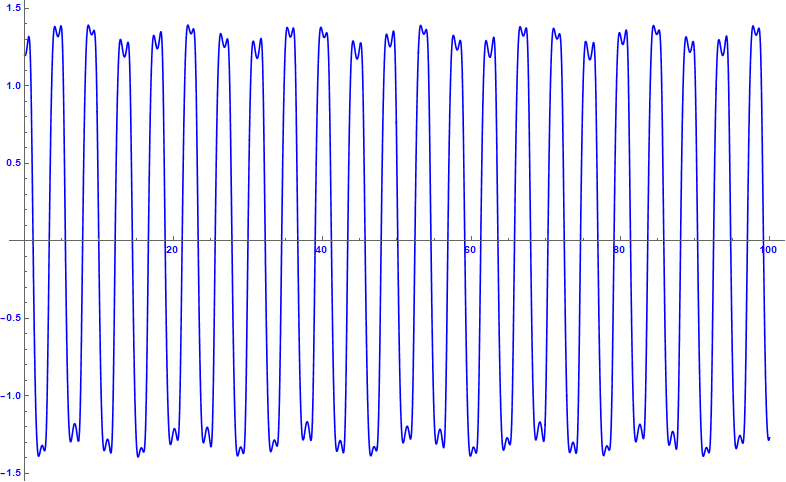}
            \end{minipage}
        }
    \caption{When $\mu=-0.1$, the system \eqref{3.1} takes the initial value $(x_{0},y_{0})=(0,1.2)$as the oscillation diagram of the change of $x$and $y$with time $t$, respectively}\label{F4}
\end{figure}

\section{Conclusions and Prospects}
\par
In this paper, we mainly study the dynamic properties of van der Pol-Duffing unforced systems with quintic terms and the existence of periodic and quasi-periodic solutions for systems with periodic external forces. Through analysis, it is found that the unforced system\eqref{1.2.3} has rich dynamic properties. The unforced system\eqref{1.2.3} has a total of three parameters, and there may be 1 to 3 equilibria in the finite plane: $O(0,0)$, $E(x_{1},0)$, $E(x_{2},0)$, where the origin exists for any value of the argument. Choosing different parameters as bifurcation will produce different bifurcation at different equilibrium. By numerical simulation, we find that the autonomous system\eqref{1.2.3} can produce pitchfork bifurcation, Hopf bifurcation and homoclinic orbit bifurcation, and the numerical simulation results are strictly proved theoretically.

Pitchfork bifurcation and Hopf bifurcation are local partial bifurcation, and the higher order term does not affect the bifurcation phase diagram. Through the stability analysis of the equilibria, it is found that if $\beta$ is regarded as the branching parameter, pitchfork bifurcation will be generated near the origin $O$. If $\mu$ is regarded as a bifurcation parameter, the Hopf bifurcation will be generated near two equilibria, $E_{1,2}$. In the analysis of pitchfork bifurcation, the center direction and non-center direction are obtained by reversible linear transformation of coordinate variables, and then the central manifold of the system\eqref{1.2.3} is calculated to limit the system to the central manifold and the standard form of pitchfork bifurcation is obtained. It is found that the system\eqref{1.2.3} generates a subcritical pitchfork bifurcation when $\mu>0$, and a supercritical pitchfork bifurcation when $\mu<0$. When analyzing Hopf bifurcation, since the system\eqref{1.2.3} is symmetric, we only need to analyze the Hopf bifurcation generated by an equilibrium. Firstly, the equilibrium point $E_{2}$ is moved to the origin by coordinate translation transformation. Then the standard form of the Hopf bifurcation is obtained by calculating the first Lyapunov coefficient and Normal-Form method. It is found that when $\mu>\mu_{c}$, the system \eqref{1.2.3} branches off a stable small limit cycle near the equilibrium points $E_{1}$ and $(E_{2})$, respectively.

Homoclinic orbit bifurcation is a global bifurcation. The system\eqref{1.2.3} is transformed into an approximate Hamilton system by transformation, and then the Melnikov function is calculated to obtain the homoclinic orbit when $\mu=\mu_{3}$, and the homoclinic orbit bifurcation line is obtained. Next, we analyze the homoclinic orbit generated by the system\eqref{1.2.3} by mapping method, and prove the existence of the homoclinic orbit again from another angle, and prove the existence of the large limit cycle (including three finite equilibria) by using the ring domain theorem. Combined with numerical simulation, it can be seen that there is only one large limit cycle and it is stable. In the global analysis, we focus on when the limit cycle of the system occurs. By calculating the stability of equilibria at infinity, we know that when $\epsilon>0$, $\beta=0$, $\mu\geq0$, the unforced system\eqref{1.2.3} has a stable limit cycle. Then, through the Bendion-Dulac criterion and index theory, we know that when $\epsilon\leq0$, $\beta>0$ and $\epsilon<0$, $\beta=0$, system\eqref{1.2.3} haven't limit cycle, and there is no limit cycle for $\mu\leq-\frac{1}{4}$. The global phase diagram of the system\eqref{1.2.3} is given based on the above analysis.

Finally, the existence of quasi-periodic solutions of system\eqref{1.2.2} with periodic disturbances is analyzed by KAM theorem. We reduce the system \eqref{1.2.2} to normal form in polar coordinates by a shift transformation of the parameters and a series of coordinate transformations similar to the canonical form of the Hopf bifurcation. It is proved by KAM theorem that there are quasi-periodic solutions with two fundamental frequencies for this canonical form, and thus the existence of quasi-periodic solutions caused by Hopf bifurcation near the equilibrium points $E_{1}$ and $E_{2}$ of the original system\eqref{1.2.2} is proved. There is no theoretical analysis and numerical simulation in the reference\cite{L8} for the above analysis.
\par
In this paper, the system\eqref{1.2.2} has been systematically analyzed, but there is still some work to be further studied on the unforced system \eqref{1.2.3}, mainly in the following two aspects:
\par
1. The existence of the large limit cycle of the system\eqref{1.2.3} when $\mu<0$ has not been discussed, we obtain from the Bendion-Dulac criterion that the system\eqref{1.2.3} has no closed orbit when $\mu\leq-\frac{1}{4}$ has no closed orbit. It is speculated that the existence of the limit cycle of the system \eqref{1.2.3} when $\mu<0$ may be the same as the case of $\mu\leq-\frac{1}{4}$, but no theoretical proof is given.
\par
2. When analyzing the existence and uniqueness of the large limit cycle of the $\epsilon>0,~\beta>0,~\mu>\zeta(\mu)$system \eqref{1.2.3}, the uniqueness of the large limit cycle is obtained by combining numerical simulation. This is partly because of the difficulty of proof, no rigorous theoretical proof has been given.

\vskip 0.2in \noindent{\bf Acknowledgments} \vskip 0.2in
\par  This work is supported  by the NNSF(11371132) of China, by Key Laboratory of High Performance Computing and Stochastic Information Processing.\\

\vskip 0.2in \noindent{\bf References} \vskip 0.2in
\newcounter{cankao}
\begin{list}
{[\arabic{cankao}]}{\usecounter{cankao}\itemsep=0cm} \small

\bibitem{L1}B. van der Pol. The nonlinear theory of electric oscillations. Proceedings of the Institute of Radio Engineers, 1934, 22(9):1051-1086.
\bibitem{L2}M. L. Cartwbight. Balthazar van der Pol. Journal of the London Mathematical Society, 1960, 35(3):367-376.
\bibitem{L4}M. Van Dyke. Perturbation Methods in Fluid Mechanics. Academic Press, 1964.
\bibitem{L6}A. H. Nayfeh and D. T. Mool. Nonlinear Oscillations. New York: John Wiley and Sons, 1979.
\bibitem{L5}Y.Yang. KBM method of analyzing limit cycle flutter of a wing with an external store and comparison with a wind-tunnel test. Journal of Sound and Vibration, 1995, 187(2):271-280.

\bibitem{L3}F. Battelli, K. J. Palmer. Chaos in the Duffing equation. Journal of Differential Equations, 1993, 101:276-301.

\bibitem{L28}I. Lopes, D. Passos, M. Nagy, et al. Oscillator models of the solar cycle: Towards the development of inversion methods. Physics, 2015, 186(1-4):535-559.
\bibitem{L29}Y. Qian, W. Zhang, B.Lin, et al. Analytical approximate periodic solutions for two-degree-of-freedom coupled van der Pol-Duffing oscillators by extended homotopy analysis method. Acta Mechanica, 2011, 219:1-14.
\bibitem{L30}P. Kumar, A. Kumar, S. Erlicher. A modified hybrid van der Pol-Duffing-Rayleigh oscillator for modelling the lateral walking force on a rigid floor. Physica D: Nonlinear Phenomena, 2017, 358:1-14.
\bibitem{L24}P. Yu, A. C. Murza, et al. Coupled oscillatory systems with D-4 symmetry and application to van der Pol oscillators. International Journal of Bifurcation and Chaos in Applied Sciences and Engineering, 2016, 26(8):1650141.
\bibitem{L25}Y. Xu, A. C. J. Luo. Independent period-2 motions to chaos in a van der Pol-Duffing oscillator. International Journal of Bifurcation and Chaos, 2020, 2030045.
\bibitem{L26}D. Monsivais-Velazquez, K. Bhattacharya, R. A. Barrio, et al. Dynamics of hierarchical weighted networks of van der Pol oscillators[J]. Chaos, 2020, 30(12):123146.
\bibitem{L15}F. Dumortier, C. Rousseau. Cubic Lienard equations with linear damping[J]. Nonlinearity, 1990, 3:1015-1039.
\bibitem{L27}F. Dumortier, C. Herssens. Polynomial Li\'{e}nard equations near infinity[J]. Journal of Differential Equations, 1999, 153:1-29.
\bibitem{L16}H.Chen, X.Chen. Dynamical analysis of a cubic Li\'{e}nard system with global parameters. Nonlinearity, 2015, 28:3535-3562.
\bibitem{L13}M. R. Cndido, J. Llibre, C. Valls. Non-existence, existence, and uniqueness of limit cycles for a generalization of the Van der Pol-Duffing and the Rayleigh-Duffing oscillators. Physica D: Nonlinear Phenomena, 2020, 407:132458.

\bibitem{L17}H.Chen, M. Han, Y. Xia. Limit cycles of a Li\'{e}nard system with symmetry allowing for discontinuity. Journal of Mathematical Analysis and Applications, 2018, 468:799-816.
\bibitem{L31}H.Chen, H. Yang, R. Zhang. Global dynamics of two classes of Li\'{e}nard systems with $Z_{2}$-equivariance. Mathematical theory and application, 2020, 40(4):1-14.
\bibitem{L14}Z. Wang, H. Chen. A nonsmooth van der Pol-Duffing oscillator (I): The sum of indices of equilibria is 1. Discrete and Continuous Dynamical Systems-B, 2021, 27(3):1549-1589.

\bibitem{L18}J. Kengne, J. C. Chedjou, M. Kom, et al. Regular oscillations, chaos, and multistability in a system of two coupled van der Pol oscillators: numerical and experimental studies. Nonlinear Dynamics, 2014, 76(2):1119-1132.
\bibitem{L19}H. Zhao, Y.Lin, Y.Dai. Hidden attractors and dynamics of a general autonomous van der Pol-Duffing oscillator*. International Journal of Bifurcation and Chaos in Applied Sciences and Engineering, 2014, 24(6):1450080.
\bibitem{L20}K. Rajagopal, A. J. M. Khalaf, Z. Wei, et al. Hyperchaos and coexisting attractors in a modified van der Pol-Duffing oscillator. International Journal of Bifurcation and Chaos, 2019, 1950067.
\bibitem{L8}Q. Han, W. Xu, X. Yue. Global bifurcation analysis of a Duffing-van der Pol oscillator with parametric excitation. International Journal of Bifurcation and Chaos, 2014, 24(4):333-126.

\bibitem{L10}S.Cai, X.Qiang. Introduction to Qualitative Theory of Ordinary Differential Equations. Higher Education Press. 1994, (in Chinese).
\bibitem{L11}Z.Zhang, T.Ding, W.Huang. Qualitative Theory of Differential Equations. Science Press. 1985, (in Chinese).
\bibitem{L23}F. Dumortier, J. Llibre, C. Joan. Qualitative Theory of Planar Differential Systems, Springer-Verlag Berlin Heidelberg, 2006.

\bibitem{L37}B.Liu, L.Huang. Existence of periodic solutions for nonlinear $n$th order ordinary differential equations. Acta Mathematica Sinica, 2004, 47(6):1133-1140.

\bibitem{L32}A. N. Kolmogorov. On conservation of conditionally periodic motions for a small change in Hamilton's function. Doklady Akademii nauk SSSR, 1954, 98: 527-530.
\bibitem{L33}V. I. Arnold. Proof of a theorem of A. N. Kolmogorov on the invariance of quasi-periodic motions under small perturbations of the Hamiltonian. Russian Mathematical Surveys, 2007, 18(5):9-36.
\bibitem{L34}J. Moser. On invariant curves of area-preserving mappings of an annulus. Matematika, 1962.
\bibitem{L35}X. Li, Z. Shang. On the existence of invariant tori in non-conservative dynamical systems with degeneracy and finite differentiability. Discrete and Continuous Dynamical Systems, 2019, 39(7):4225-4257.
\bibitem{L36}H. R\"{u}ssmann. Invariant tori in non-degenerate nearly integrable Hamiltonian systems. Regular and Chaotic Dynamics, 2001, 6(2):119-204.

\bibitem{L9}Y. A. Kuznetsov. Elements of applied bifurcation theory. Applied Mathematical Sciences, 2004, 288(2): 715-730.
\bibitem{L21}C. Li, C. Rousseau. Codimension 2 symmetric homoclinic bifurcations and application to 1:2 resonance. Canadian Journal of Mathematics, 1990, 42(2): 279-287.
\bibitem{L22}W. Zhang, P. Yu. Degenerate bifurcation analysis on a parametrically and externally excited mechanical system. International Journal of Bifurcation and Chaos in Applied Sciences and Engineering, 2001, 689-709.

\end{list}

\end{document}